\documentclass[11pt]{article}
\usepackage{a4,fullpage,epsf,psfrag}
\usepackage{enumerate}
\usepackage{epsfig}
\usepackage{graphicx}
\usepackage{graphics}
\usepackage{latexsym}
\usepackage{epsfig}
\usepackage{amsfonts}
\usepackage{amsthm}
\usepackage{amsmath}
\usepackage{amssymb}

\usepackage[all]{xy}

\usepackage[
            ps2pdf,
            breaklinks=true,
            colorlinks=true]{hyperref}

\usepackage{amscd}
\usepackage{mathrsfs}
\usepackage{stmaryrd}
\usepackage{amsopn}

\newtheorem{nostheorem}{Theorem}
\newtheorem*{notheorem}{Connectivity Theorem (Atiyah)}
\newtheorem*{no4theorem}{Convexity Theorem (Atiyah and Guillemin-Sternberg)}
\newtheorem{theorem}{Theorem}[section]
\newtheorem{prop}[theorem]{Proposition}

\newtheorem{lemma}[theorem]{Lemma}

\newcommand{\R}{{\mathbb R}}

\newcommand{\Z}{{\mathbb Z}}

\newcommand{\T}{{\mathbb T}}
\newcommand{\norm}[1]{\| #1\|}

\newcommand{\phy}{\varphi}

\newcommand{\op}[1]{\!\!\mathop{\rm ~#1}\nolimits}
\newcommand{\DD}{\mathrm{d}}

\newenvironment{remark}{\refstepcounter{theorem}\par\medskip\noindent{\bf
Remark~\thetheorem~~}}{\unskip\nobreak\hfill\hbox{ $\oslash$}\par\bigskip}

\newenvironment{example}{\refstepcounter{theorem}\par\medskip\noindent{\bf
Example~\thetheorem~~}}{\unskip\nobreak\hfill\hbox{ $\oslash$}\par\bigskip}

\newenvironment{definition}{\refstepcounter{theorem}\par\medskip\noindent{\bf
Definition~\thetheorem~~}}

\newcommand{\deriv}[2]{\frac{\partial #1}{\partial #2}}

\newcommand{\abs}[1]{\left|#1\right|}

\newcommand{\impliq}{\Rightarrow}
\newcommand{\fleche}{\rightarrow}

\newcommand{\RM}{\mathbb{R}}

\newcommand{\NM}{\mathbb{N}}

\renewcommand{\geq}{\geqslant}
\renewcommand{\leq}{\leqslant}

\begin{document}

\title{{\bf Symplectic bifurcation theory for integrable systems}}

\author{\'Alvaro Pelayo, Tudor S. Ratiu, and San V\~{u} Ng\d{o}c}

\date{}

\maketitle

\begin{abstract}
  This paper develops a symplectic bifurcation theory for integrable
  systems in dimension four.  We prove that if an integrable system
  has no hyperbolic singularities and its bifurcation diagram has no
  vertical tangencies, then the fibers of the induced singular Lagrangian
  fibration are connected. The image of this
  singular Lagrangian fibration is, up to smooth deformations, a planar region bounded by
  the graphs of two continuous functions. The bifurcation diagram consists
  of the boundary points in this image  plus a countable
  collection of  rank zero singularities, which are contained in the
  interior of the image.   Because it recently has
  become clear to the mathematics and mathematical physics communities
  that the bifurcation diagram of an integrable system provides the
  best framework to study symplectic invariants, 
  this paper provides a setting for studying quantization questions, and
  spectral theory of quantum integrable systems.
  \end{abstract}

\section{Introduction and Main Theorems} \label{sec:intro}

A major
obstacle to a symplectic theory of finite dimensional integrable
Hamiltonian systems is that \emph{differential topological} and
\emph{symplectic} problems appear side by side, but smooth and
symplectic methods do not always mesh well. Morse-Bott theory
represents a success in bringing together in a cohesive way continuous
and differential tools, and it has been used effectively to study
properties of dynamical systems. But incorporating symplectic
information into the context of dynamical systems is far from
automatic. However, many concrete examples are known for which
computations, and numerical simulations, exhibit a close relationship
between the symplectic dynamics of a system, and the differential
topology of its bifurcation set. 

In the 1980s and 1990s, the Fomenko
school developed a Morse theory for regular energy surfaces of
integrable systems. Moreover, theoretical successes (in any dimension)
for compact periodic systems in the 1970s and 1980s by Atiyah,
Guillemin, Kostant, Sternberg and others, gave hope that one can find
a mathematical theory for bifurcations of integrable systems in the
symplectic setting.

This paper develops a symplectic bifurcation theory for integrable
systems in dimension four -- compact or not.  Because it recently has
become clear that the bifurcation diagram of an integrable system is
the natural setting to study symplectic invariants (see for instance
\cite{PeVN2009, PeVN2010}), this paper provides a setting for the
study of quantum integrable systems.  Semiclassical quantization is a
strong motivation for developing a systematic bifurcation theory of
integrable systems; the study of bifurcation diagrams is fundamental
for the understanding of quantum spectra~\cite{CDV,lazutkin, Ze1998}.
Moreover, the results of this paper may have applications to mirror
symmetry and symplectic topology because an integrable system without
hyperbolic singularities gives rise to a toric fibration with
singularities. The base space is endowed with a singular integral
affine structure.  These singular affine structures are studied in
symplectic topology, mirror symmetry, and algebraic geometry; for
instance, they play a central role in the work of Kontsevich and
Soibelman \cite{KS}. We refer to Section \ref{sec:remarks} for further
analysis of these applications, as well as a natural connection to the
study of solution sets in real algebraic geometry.

The
development of the theory requires the introduction of methods to
construct Morse-Bott functions which, from the point of view of
symplectic geometry, behave well near the singularities of integrable
systems.  These methods use  Eliasson's
theorems on linearization of non-degenerate singularities of
integrable systems, and the symplectic topology of integrable systems, to which
many have contributed.

The first part of this paper is concerned with the connectivity of
joint level sets of vector-valued maps on manifolds, when these are
defined by the components of an integrable system.  The most striking
previous result in this direction is Atiyah's 1982 theorem which
guarantees the connectivity of the fibers of the momentum map when the
integrable system comes from a Hamiltonian torus action. The second
part of the paper explains how the pioneering results proven in the
seventies and eighties by Atiyah, Guillemin, Kostant, Kirwan, and
Sternberg describing the image of the momentum of a \emph{Hamiltonian
  compact group action} also hold in the context of integrable systems
on four-dimensional manifolds, when there are no hyperbolic
singularities.  The conclusions of the theorems in this paper are
essentially optimal. Moreover, there is only one transversality
assumption on the integrable system: that there should be no vertical
tangencies on the bifurcation set, up to diffeomorphism. If this
condition is violated then there are examples which show that one
cannot hope for any fiber connectivity.

The work of Atiyah, Guillemin, Kostant, Kirwan, and Sternberg
exhibited connections between symplectic geometry, combinatorics,
representation theory, and algebraic geometry.  Their work guarantees
the convexity of the image of the momentum map (intersected with the
positive Weyl chamber if the compact group is
non-commutative). Although this property no longer holds for general
integrable systems, an explicit description of the image of the
singular Lagrangian fibration given by an integrable system with two
degrees of freedom can be given. The understanding of this image,
which corresponds to the bifurcation diagram of the dynamics in the
physics literature, is essential for the description of the system, as
it has proven to be the best framework to define new symplectic
invariants of integrable systems, and hence to quantize; see the
recent work on semitoric integrable systems 
\cite{PeVN2009,  PeVN2010, PeVN2011, PeVNpress}.

\subsection*{Fiber connectivity for integrable systems}

The most striking known result for fiber connectivity of vector-valued
functions is Atiyah's famous connectivity theorem \cite{At1982} proved
in the early eighties. 

\begin{notheorem} 
  Suppose that $(M,\, \omega)$ is a compact, connected, symplectic,
  $2m$-dimensional manifold.  For smooth functions $f_1,\, \ldots,\,
  f_n\colon M \to \mathbb{R}$, let $\varphi_i^{t_i}$ be the flow of
  the Hamiltonian vector field $\mathcal{H}_{f_i}$, where
  $\mathcal{H}_{f_i}$ is defined by the equation
  $\omega(\mathcal{H}_{f_i},\, \cdot)=\DD f_i$.  We denote by $\T^n$ the
  $n$-dimensional torus $\R^n/\Z^n$. 
  Suppose that $M\ni p
  \mapsto (\varphi_1^{t_1} \circ \ldots \circ \varphi_n^{t_n})(p) \in
  M $, where $(t_1, \ldots, t_n) \in \RM^n$, defines a $\T^n$-action
  on $M$.  Then the fibers of the map $F:=(f_1,\,\ldots,\, f_n) \colon
  M \to \mathbb{R}^n$ are connected.
\end{notheorem}

This theorem has been generalized by a number of authors to general
compact Lie groups actions and more general symplectic
manifolds. Indeed, a Hamiltonian $m$-torus action on a $2m$-manifold
may be viewed as a very particular integrable system. A long standing
question in the integrable systems community has been to what extent
Atiyah's result holds for integrable systems, where checking fiber
connectivity by hand is extremely difficult (even for easily
describable examples). This paper answers this question in the
positive in dimension four: if an integrable system has no hyperbolic
singularities and its bifurcation diagram has no vertical tangencies,
then the fibers of the integrable system are connected.

To state this result precisely, recall that a map $F=(f_1,\ldots,f_n)
\colon (M,\,\omega) \to \mathbb{R}^n$ is a \emph{integrable
  Hamiltonian system} if $\mathcal{H}_{f_1},\, \ldots,
\mathcal{H}_{f_n}$ are point-wise almost everywhere linearly
independent and for all indices $i,\,j$, the function $f_i$ is
invariant along the flow of the Hamiltonian vector field
$\mathcal{H}_{f_j}$.  (Recall: $\mathcal{H}_{f_i}$ is the vector field
defined by $\omega(\mathcal{H}_{f_i},\, \cdot)=\DD f_i$).

In general, fiber connectivity is no longer true for integrable
systems, due to the existence of singularities. For instance, consider
the manifold is $M = S^2 \times S^1 \times S^1$.  Choose the following
coordinates on $S^2$: $h \in [1,\,2]$ and $a \in S^1=
\mathbb{R}/2\pi\mathbb{Z}$.  Choose coordinates $b, c \in \mathbb{R}/
2\pi\mathbb{Z}$ on $S^1 \times S^1$ and let $\DD h \wedge \DD a + n\DD
b\wedge\DD c$ be the symplectic form on $M$, where $n$ is a positive
integer. The map $F(h, \,a,\,b,\,c) = (h \cos(n b) , \,h \sin(n b) )$,
defines an integrable system. The fiber over any regular value of $F$
is $n$ copies of $S^1\times S^1$.  The image $F(M)$ is an annulus (see
Figure~\ref{fig:annulus-only}).

\begin{figure}[h]
  \centering
  \includegraphics[height=4.5cm, width=4.5cm]{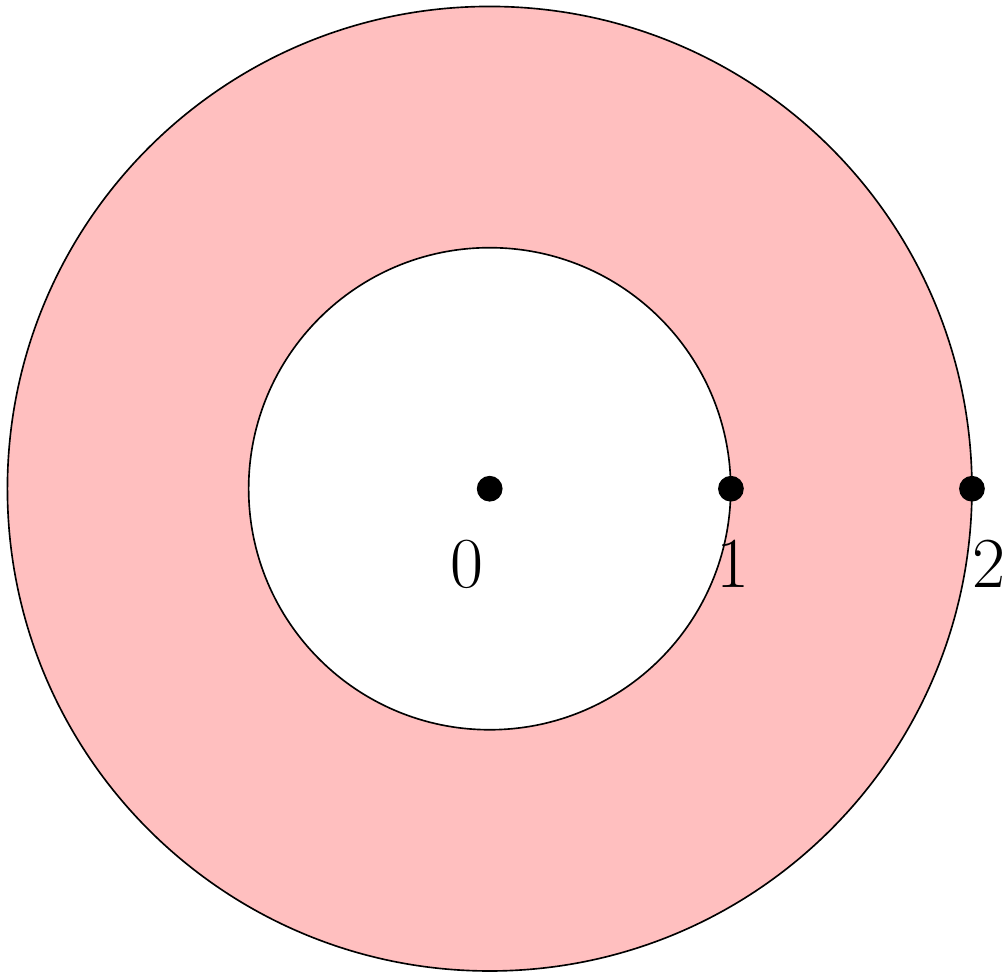}
  \caption{Image $F(M)$ of integrable system with disconnected fibers,
    on the compact manifold $S^2 \times S^1 \times S^1$. This
    integrable system has only singularities of elliptic and
    transversally elliptic type, so it is remarkable how while being
    quite close to be a toric system, fiber connectivity is lost.}
  \label{fig:annulus-only}
\end{figure}

An annulus has the property that it has vertical tangencies,
and it cannot be deformed into a domain without such
tangencies. Remarkably, if such tangencies do not exist in some
deformation of $F(M)$, fiber connectivity still holds. This is for
instance the case for Hamiltonian torus actions. Next we state this
precisely.

In this paper, manifolds are assumed to be ${\rm C}^{\infty}$ and second
countable. Let us recall here some standard definitions. 
A map $f \colon X \to Y$ between topological spaces is \textit{proper}
if the preimage of every compact set is compact.  Let $X$, $Y$ be
smooth manifolds, and let $A\subset X$. A map $f:A\to Y$ is said to be
\emph{smooth} if at any point in $A$ there is an open neighborhood on
which $f$ can be smoothly extended. The map $f$ is called a
\emph{diffeomorphism onto its image} when $f$ is injective, smooth,
and its inverse $f^{-1}:f(X)\to X$ is smooth.  If $X$ and $Y$ are
smooth manifolds, the \emph{bifurcation set} $\Sigma_f$ of a smooth
map $f: X\rightarrow Y$ consists of the points of $X$ where $f$ is
not locally trivial (see Definition \ref{def:bifurcation}). It is 
known that the set of critical values of $f$ is included in the
bifurcation set and that if $f$ is proper this inclusion is an
equality (see \cite[Proposition 4.5.1]{AbMa1978} and the comments
following it).
 
Second, recall that an integrable system $F \colon M \to\mathbb{R}^2$
is called \emph{non-degenerate} if its singularities are
non-degenerate (see Definition \ref{def:nondegpoi}). If $F$ is proper
and non-degenerate, then $\Sigma_F$ is the image of a piecewise smooth
immersion of a 1-dimensional manifold
(Proposition~\ref{prop:strata}). We say that a vector in $\RM^2$ is
\emph{tangent to $\Sigma_F$} whenever it is directed along a left
limit or a right limit of the differential of the immersion.  We say
that the curve $\gamma$ has a \emph{vertical tangency} at a point $c$
if there is a vertical tangent vector at $c$. Our first main result is
the following.

\begin{nostheorem}[Connectivity for Integrable Systems -- Compact
  Case] \label{theo:main-connectivity0} Suppose that $(M, \omega)$ is
  a compact connected symplectic four-manifold. Let $F \colon M
  \to\mathbb{R}^2$ be a non-degenerate integrable system without
  hyperbolic singularities. Denote by $\Sigma_F$ the bifurcation set
  of $F$.  Assume that there exists a diffeomorphism $g \colon F(M)
  \to \mathbb{R}^2$ onto its image such that $g(\Sigma_F)$ does not
  have vertical tangencies (see Figure~\ref{fig:diffeo}).  Then $F$
  has connected fibers.
\end{nostheorem}

\begin{figure}[h]
  \centering
  \includegraphics[width=0.5\textwidth]{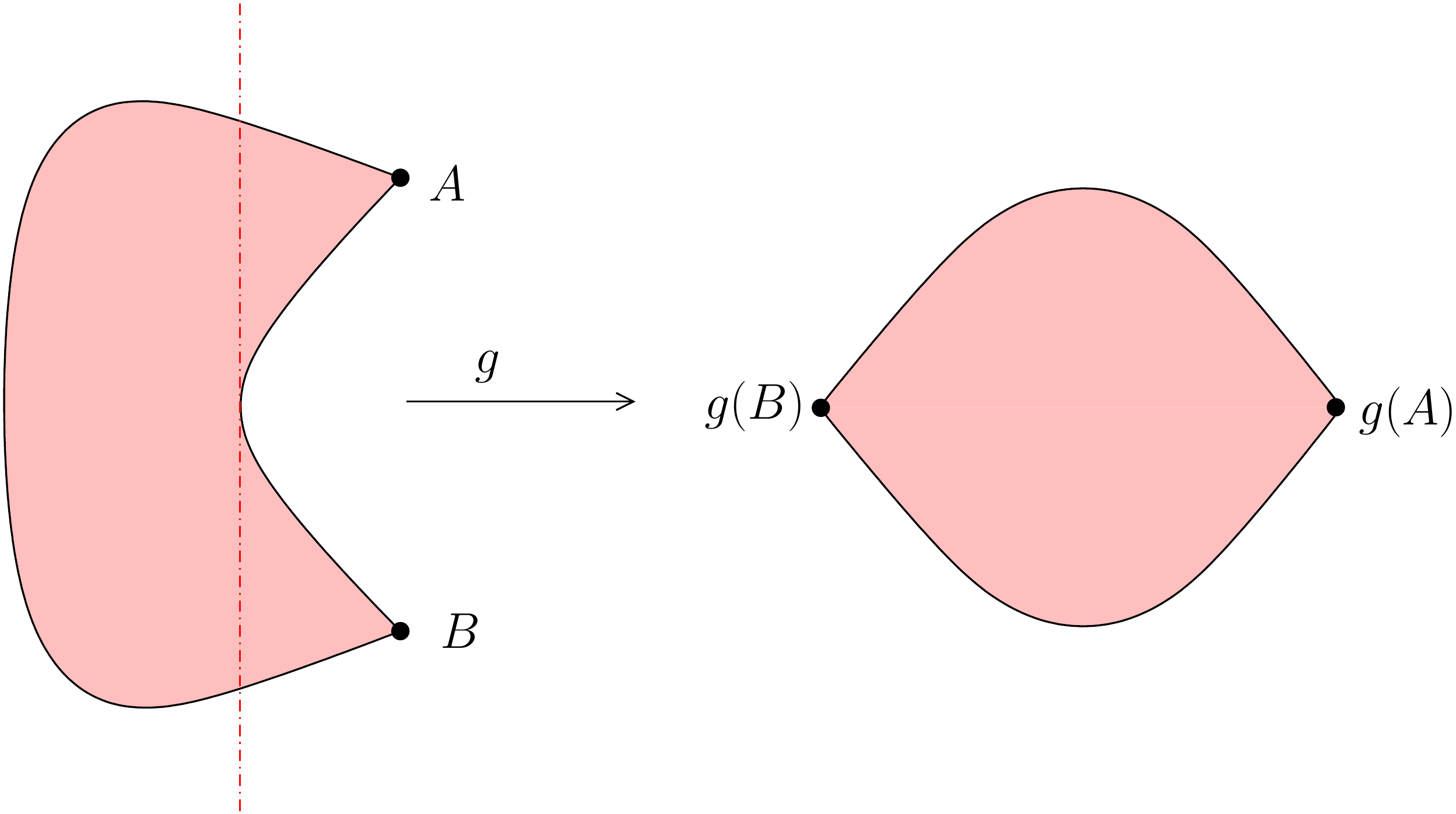}
  \caption{Suppose that the bifurcation set $\Sigma_F$ of $F$ consists
    precisely of the boundary points in the left figure (which depicts
    $F(M)$). The diffeomorphism $g$ transforms $F(M)$ to the region on
    the right hand side of the figure, in order to remove the original
    vertical tangencies on $\Sigma_F$.}
  \label{fig:diffeo}
\end{figure}

\begin{remark}
  If $F \colon M \to \mathbb{R}^2$ in Theorem
  \ref{theo:main-connectivity} is the momentum map of a Hamiltonian
  $2$\--torus action then $\Sigma_F=\partial (F(M))$. This is no
  longer true for general integrable systems; the simplest example of
  this is the spherical pendulum, which has a point in the bifurcation
  set in the interior of $F(M)$.
\end{remark}

We denote by $C_{\alpha,\beta}$ the cone in Figure \ref{fig:cone},
i.e., the intersection of the half-planes defined by $y \geq
(\tan\alpha)\, x$ and $y \leq (\tan \beta)\, x$ on the plane
$\mathbb{R}^2$. This cone will be called \emph{proper}, if $\alpha>0,
\, \beta>0$, $\alpha+\beta<\pi$. Theorem \ref{theo:main-connectivity}
can be extended to non-compact manifolds as follows.

\begin{nostheorem}[Connectivity for Integrable Systems -- Non-compact
  Case] \label{theo:main-connectivity} Suppose that $(M, \omega)$ is a
  connected symplectic four-manifold. Let $F \colon M \to
  \mathbb{R}^2$ be a non-degenerate integrable system without
  hyperbolic singularities such that $F$ is a proper map. Denote by
  $\Sigma_F$ the bifurcation set of $F$.  Assume that there exists a
  diffeomorphism $g \colon F(M) \to \mathbb{R}^2$ onto its image such
  that:
  \begin{enumerate}[\upshape(i)]
  \item the image $g(F(M))$ is included in a proper convex cone
    $C_{\alpha,\,\beta}$ (see Figure~\ref{fig:cone});
  \item the image $g(\Sigma_F)$ does not have vertical tangencies (see
    Figure~\ref{fig:diffeo}).
  \end{enumerate}
  Then $F$ has connected fibers.
\end{nostheorem}

Note that Theorem \ref{theo:main-connectivity} clearly implies Theorem
\ref{theo:main-connectivity0}.

\begin{figure}[h]
  \centering
  \includegraphics[width=0.3\textwidth]{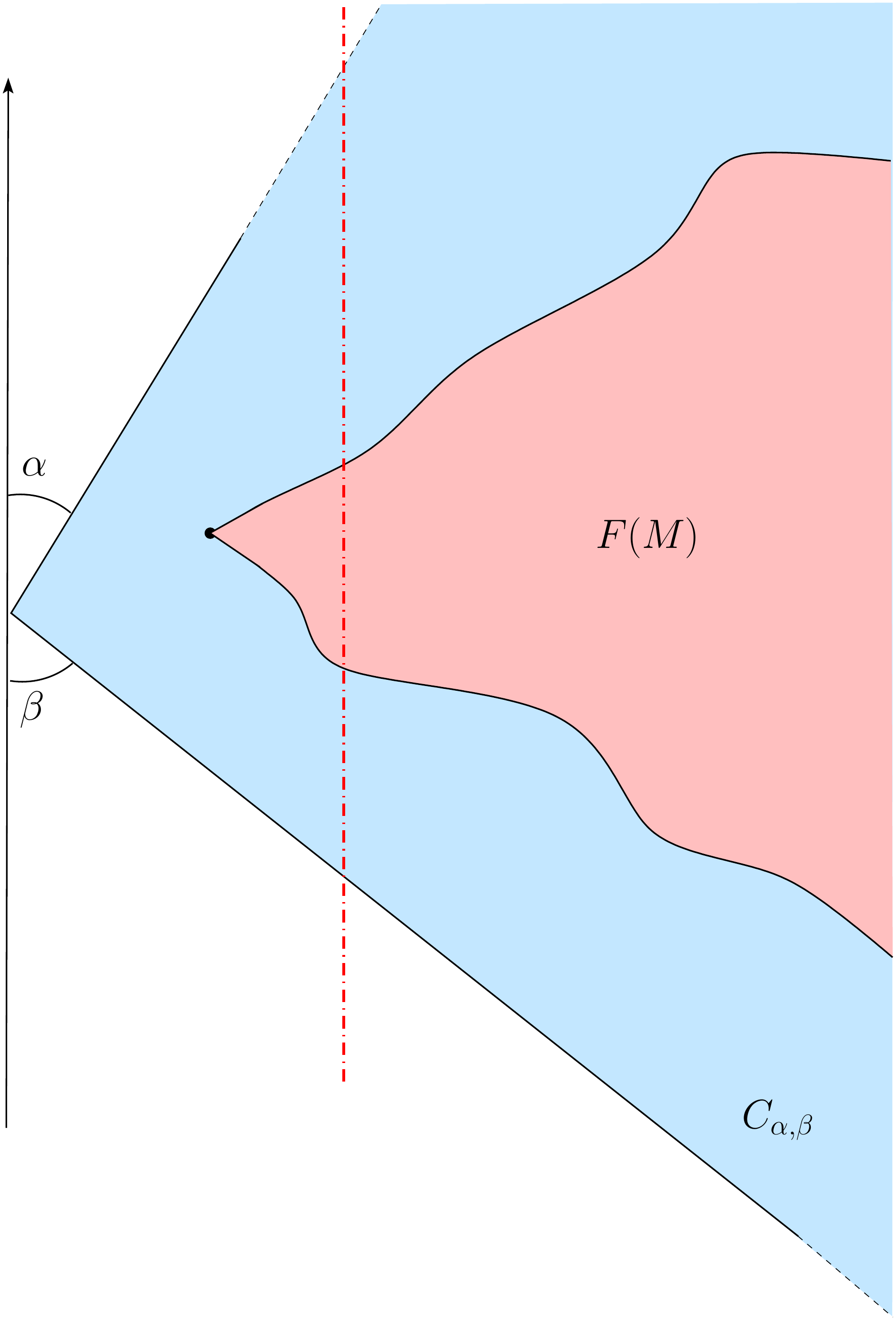}
  \caption{The image $F(M)$ lies in the convex cone $C_{\alpha,\beta}$
    and has no vertical tangencies. See condition~1 in
    theorem~\ref{theo:main-connectivity}.}
  \label{fig:cone}
\end{figure}

\begin{remark}
  The assumption in Theorem \ref{theo:main-connectivity} is optimal in
  the following sense: if there exist vertical tangencies then the
  system can have disconnected fibers (Examples \ref{ex0}, \ref{ex2});
  see Theorem \ref{theo:striking}. Other than these exceptions we do
  not know of an integrable system with disconnected fibers and which
  does not violate our assumptions.
\end{remark}

Further in this article, we introduce a weaker transversality
condition that allows us to deal with some cases of vertical
tangencies. Using this condition and
Theorem~\ref{theo:main-connectivity}, we will prove the following.

\begin{figure}[h]
  \centering
  \includegraphics[width=0.55\textwidth]{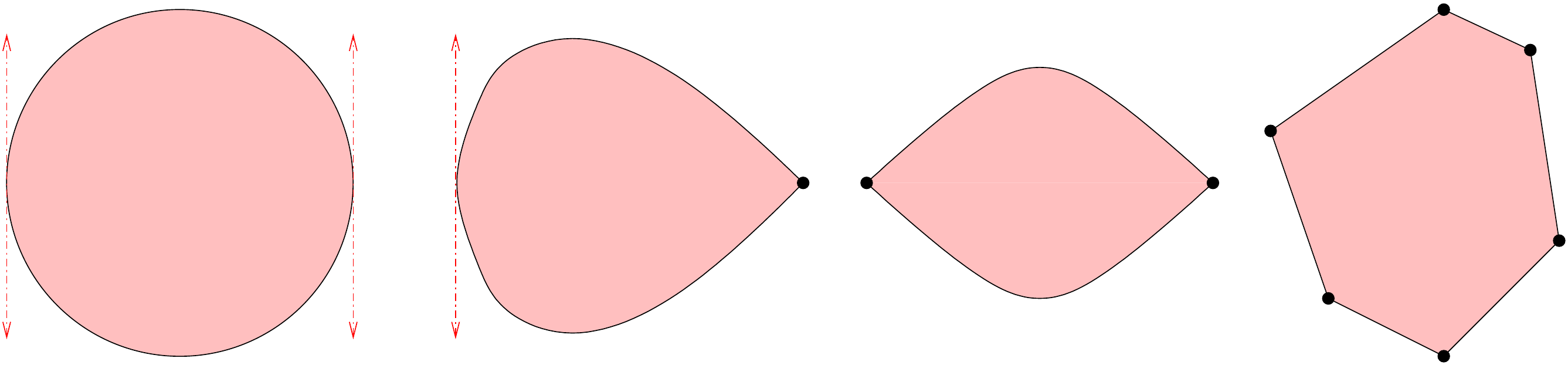}
  \caption{A disk, a disk with a conic point, a disk with two conic
    points, or a compact convex polygon.}
  \label{fig:4images}
\end{figure}

\begin{nostheorem} \label{theo:striking} Suppose that $(M, \omega)$ is
  a compact connected symplectic four-manifold. Let $F \colon M \to
  \mathbb{R}^2$ be a non- degenerate integrable system without
  hyperbolic singularities.  Assume that
  \begin{enumerate}
  \item[{\rm (a)}] the interior of $F(M)$ contains a finite number of
    critical values;
  \item[{\rm (b)}] there exists a diffeomorphism $g$ such that
    $g(F(M))$ is either a disk, a disk with a conic point, a disk with
    two conic points, or a compact convex polygon (see Figure
    \ref{fig:4images}).
  \end{enumerate}
  Then the fibers of $F$ are connected.
\end{nostheorem}

Here a neighborhood of a \emph{conic point} is by definition locally
diffeomorphic to some proper cone $C_{\alpha,\beta}$.

\subsection*{Structure of the image of an integrable system}

Atiyah proved his connectivity theorem \cite{At1982} simultaneously
with the so called convexity theorem of Atiyah, Guillemin, and
Sternberg \cite{At1982,GuSt1982}; it is one of the main results in
symplectic geometry. Their convexity theorem describes the image of
the momentum map of a Hamiltonian torus action.  Altogether, this
result generated much subsequent research, in particular it led Kirwan
to prove a remarkable non-commutative version \cite{kirwan}.

\begin{no4theorem}
  Suppose that $(M,\, \omega)$ is a compact, connected, symplectic,
  $2m$-dimensional manifold.  For smooth functions $f_1,\, \ldots,\,
  f_n\colon M \to \mathbb{R}$, let $\varphi_i^{t_i}$ be the flow of
  the Hamiltonian vector field $\mathcal{H}_{f_i}$, where
  $\mathcal{H}_{f_i}$ is defined by the equation
  $\omega(\mathcal{H}_{f_i},\, \cdot)=\DD f_i$.   We denote by $\T^n$ the
  $n$-dimensional torus $\R^n/\Z^n$. Suppose that $M\ni p
  \mapsto (\varphi_1^{t_1} \circ \ldots \circ \varphi_n^{t_n})(p) \in
  M $, where $(t_1, \ldots, t_n) \in \RM^n$, defines a $\T^n$-action
  on $M$.  Then the image of $F:=(f_1,\, \ldots,\, f_n) \colon M \to
  \mathbb{R}^n$ is a convex polytope.
\end{no4theorem}

\begin{remark}
  An important paper prior to the work of Atiyah, Guillemin, Kirwan,
  and Sternberg dealing with convexity properties in particular
  instances is Kostant's \cite{Ko1974}, who also refers to preliminary
  questions of Schur, Horn and Weyl.  These convexity results were
  used by Delzant \cite{Delzant1988} in his classification of
  \emph{symplectic toric manifolds}.  All together, these papers
  revolutionized symplectic geometry and its connections to
  representation theory, combinatorics, and complex algebraic
  geometry.  For a detailed analysis of symplectic toric manifolds in
  the context of complex algebraic geometry, see Duistermaat-Pelayo
  \cite{DuPe2009}.
\end{remark}

In this paper we prove the natural version of the
Atiyah\--Guillemin\--Sternberg convexity theorem in the context of
integrable systems.  Before stating it we recall that the
\emph{epigraph $\op{epi}(f) \subseteq \mathbb{R}^{n+1}$ of a map $f
  \colon A \subseteq \mathbb{R}^n \to \mathbb{R} \cup \{\pm\infty\}$}
consists of the points lying on or above its graph, i.e., the set
$\op{epi}(f):= \{ (x,\, y) \in A \times \mathbb{R} \, | \, y \geq
f(x)\}$. Similarly, the \emph{hypograph $\op{hyp}(f)
  \subseteq\mathbb{R}^{n+1}$ of a map $f \colon A \subseteq
  \mathbb{R}^n \to\mathbb{R} \cup \{\pm \infty\}$} consists of the
points lying on or below its graph, i.e., the set $ \op{hyp}(f):= \{
(x,\, y) \in A \times\mathbb{R} \, | \, y \leq f(x) \}$.

\begin{nostheorem}[Image of Lagrangian fibration of integrable system
  -- Compact Case]
  \label{theo:main-image0}
  Suppose that $(M, \omega)$ is a compact connected symplectic
  four-manifold. Let $F \colon M \to \mathbb{R}^2$ be a non-degenerate
  integrable system without hyperbolic singularities. Denote by
  $\Sigma_F$ the bifurcation set of $F$.  Assume that there exists a
  diffeomorphism $g \colon F(M) \to \mathbb{R}^2$ onto its image such
  that $g(\Sigma_F)$ does not have vertical tangencies (see
  Figure~\ref{fig:diffeo}). Then:
  \begin{itemize}
  \item[{\rm (1)}] the image $F(M)$ is contractible and  the bifurcation
set can be described as
  $ \Sigma_F=\partial  (F(M)) \sqcup \mathcal{F}, $ where $\mathcal{F}$ is
a finite set
    of rank $0$ singularities which is contained in the interior of
    $F(M)$;
  \item[{\rm (2)}]
  let $(J,\,H):=g\circ F$ and let $J(M)=[a,\,b]$. Then  the functions
$H^{+},\,
    H^{-} \colon [a,\,b]  \to \R$ defined by $H^{+}(x) := \max_{J^{-1}(x)}
    H$ and $H^{-}(x) := \min_{J^{-1}(x)} H$ are continuous and $F(M)$
    can be described as $F(M) = g^{-1}(\op{epi}(H^{-}) \cap
    \op{hyp}(H^+))$.
\end{itemize}
\end{nostheorem}

Figure~\ref{fig:fig1} shows a possible image $F(M)$, as described in
Theorem~\ref{theo:main-image0}. In the case on non-compact manifolds 
we have the following result.

\begin{figure}[h]
  \centering
  \includegraphics[width=0.8\textwidth]{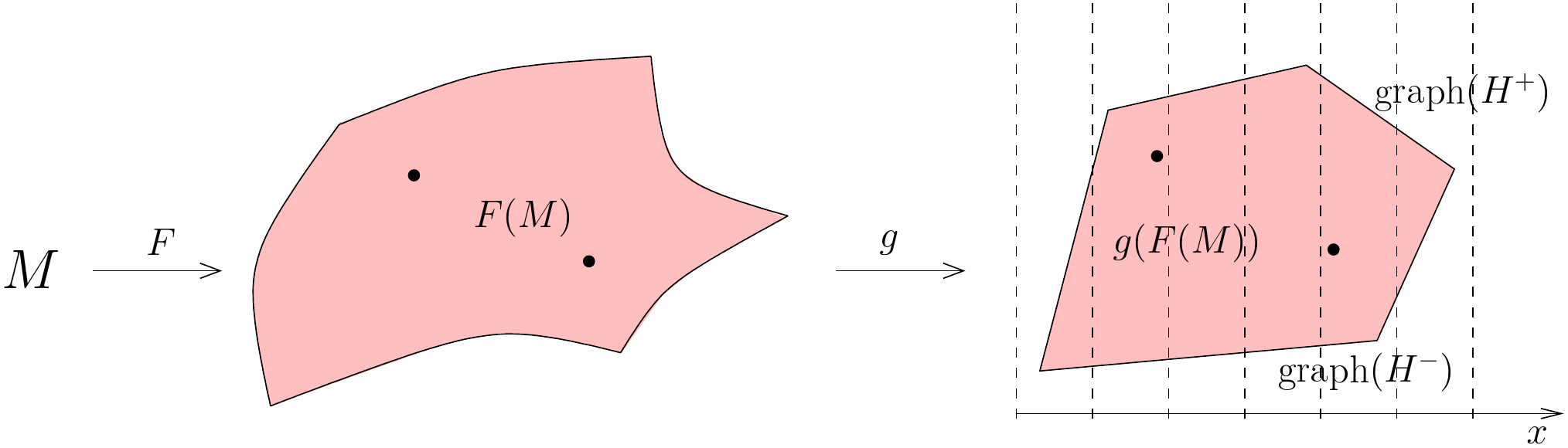}
  \caption{Description of the image of an integrable system. The image
    is first transformed to remove vertical tangencies, and then it
    can be described as a region bounded by two graphs.}
  \label{fig:fig1}
\end{figure}

\begin{nostheorem}[Image of Lagrangian fibration of integrable system
  -- Non-compact case]
  \label{theo:main-image}
  Suppose that $(M, \omega)$ is a connected symplectic
  four-manifold. Let $F \colon M \to \mathbb{R}^2$ be a non-degenerate
  integrable system without hyperbolic singularities such that $F$ is
  a proper map. Denote by $\Sigma_F$ the bifurcation set of $F$.
  Assume that there exists a diffeomorphism $g \colon F(M) \to
  \mathbb{R}^2$ onto its image such that:
  \begin{enumerate}[\upshape(i)]
  \item the image $g(F(M))$ is included in a proper convex cone
    $C_{\alpha,\, \beta}$ (see Figure~\ref{fig:cone});
  \item the image $g(\Sigma_F)$ does not have vertical tangencies (see
    Figure~\ref{fig:diffeo}).
  \end{enumerate}
  Equip $\overline{\RM}:=\mathbb{R} \cup \{\pm\infty\}$ with the
  standard topology.  Then:
  \begin{itemize}
  \item[{\rm (1)}] the image $F(M)$ is contractible and the
    bifurcation set can be described as $ \Sigma_F=\partial (F(M))
    \sqcup \mathcal{F}, $ where $\mathcal{F}$ is a countable set of
    rank zero singularities which is contained in the interior of
    $F(M)$;
  \item[{\rm (2)}] let $(J,\,H):=g\circ F$. Then the functions
    $H^{+},\, H^{-} \colon J(M) \to \R$ defined on the interval $J(M)$
    by $H^{+}(x) := \max_{J^{-1}(x)} H$ and $H^{-}(x) :=
    \min_{J^{-1}(x)} H$ are continuous and $F(M)$ can be described as
    $F(M) = g^{-1}(\op{epi}(H^{-}) \cap \op{hyp}(H^+))$.
  \end{itemize}
\end{nostheorem}

Note that Theorem \ref{theo:main-image} clearly implies Theorem
\ref{theo:main-image0}. The rest of this paper is devoted to proving
Theorem \ref{theo:main-connectivity}, Theorem \ref{theo:striking} and
Theorem \ref{theo:main-image}.
\\
\\
{\bf Acknowledgements.} We thank Denis Auroux, Thomas Baird, and
Helmut Hofer for enlightening discussions. The authors are grateful to
Helmut Hofer for his essential support that made it possible for TSR
and VNS to visit AP at the Institute for Advanced Study during the
Winter and Summer of 2011, where a part of this paper was
written. Additional financial support for these visits was provided by
Washington University in St Louis and by NSF.  The authors also thank
MSRI, the Mathematisches Forschungsinstitut Oberwolfach, Washington
University in St Louis, and the Universit\'e de Rennes I for their
hospitality at different stages during the preparation of this paper
in 2009, 2010, and 2011.

AP was partly supported by an NSF Postdoctoral Fellowship, NSF Grants
DMS-0965738 and DMS-0635607, an NSF CAREER Award, a Leibniz Fellowship,
Spanish Ministry of Science Grant MTM 2010-21186-C02-01, 
and by the CSIC (Spanish National Research Council).

TSR. was partly supported by a MSRI membership and Swiss NSF grant
200020-132410.

VNS was partly supported by the NONAa grant from the French ANR and
the Institut Universitaire de France.
\\
\\
The paper combines a number of results
to arrive at the proofs of these theorems. The following diagram
describes the structure of the paper.

$$
\xymatrix{ \text{\bf Theorem~\ref{singularities_theorem}} \ar @<-1cm>
  @/_2cm/ [dddddddd] \ar @/_1.5cm/ [ddd] \ar[d] \ar[rddd] &
  & & \\
  \text{\bf Theorem~\ref{theo:interior}} \ar @/^1.7cm/ [ddddd] &
  \text{Proposition~\ref{prop:strata}} \ar[d] & & \\
  &
  \text{\bf Theorem~\ref{prop:bott}} \ar[d] & & & \\
  \text{Lemma~\ref{lemm:semibounded}} \ar[d] &
  \text{Proposition~\ref{prop:bott-index}} \ar[d] \ar[dddrr]& & \\
  \text{Proposition~\ref{prop:levelset}} \ar @/_1.2cm/ [dd] &
  \text{Proposition~\ref{prop:J-connected}} \ar[dd] \ar @/^1cm/
  [dddd]& & \\
  \text{Lemma~\ref{tricky}} \ar[d] & & & \\
  \text{\bf Theorem~\ref{theo:fibers}} \ar[r]& \text{\bf
    Theorem~\ref{theo:main-connectivity}} \ar[r] & \text{\bf
    Theorem~\ref{theo:main-connectivity0}} \ar[r] & \text{\bf Theorem~\ref{theo:striking}}\\
  \text{Lemma~\ref{lemma:critical-set}} \ar[d] & & & \\
  \text{\bf Theorem~\ref{imagemomentum:theorem}} \ar[r] & \text{\bf
    Theorem~\ref{theo:main-image}} \ar[r] & \text{\bf
    Theorem~\ref{theo:main-image0}} }
$$

\section{Basic properties of almost-toric systems} \label{sec2}

In this section we prove some basic results that we need in of Section
\ref{fibers:sec} and Section \ref{sec3}. Let $(M, \omega)$ be a
connected symplectic 4-manifold.

\subsection*{Toric type maps}

A smooth map $F \colon M \to \mathbb{R}^2$ is \emph{toric} if there
exists an effective, integrable Hamiltonian $\T^2$-action on $M$ whose
momentum map is $F$.  It was proven in \cite{LeMeToWo1998} that if $F$
is a proper momentum map for a Hamiltonian $\T^2$-action, then the
fibers of $F$ are connected and the image of $F$ is a rational convex
polygon.

\subsection*{Almost-toric systems}
We shall be interested in maps $F: M \rightarrow \mathbb{R}^2$ that
are not toric yet retain enough useful topological properties.  In the
analysis carried out in the paper we shall need the concept of
non-degeneracy in the sense of Williamson of a smooth map from a
4-dimensional phase space to the plane.

\begin{definition} \label{def:nondegpoi} Suppose that $(M, \omega)$ is
  a connected symplectic four-manifold.  Let $F=(f_1,f_2)$ be an
  integrable system on $(M, \omega)$, and $m \in M $ a critical point
  of $F$. If $\DD_mF=0$, then $m$ is called \emph{non-degenerate} if
  the Hessians $\operatorname{Hess}f_j(m)$ span a Cartan subalgebra of
  the symplectic Lie algebra of quadratic forms on the tangent space
  $({\rm T}_m M, \omega_m)$. If
  $\operatorname{rank}\left(\DD_mF\right)=1$ one can assume that
  $\DD_m f_1 \neq0$.  Let $\iota \colon S \to M$ be an embedded local
  $2$-dimensional symplectic submanifold through $m$ such that ${\rm
    T}_m S\subset \ker(\DD_m f_1)$ and ${\rm T}_mS$ is transversal to
  $\mathcal{H}_{f_1}$. The critical point $m$ of $F$ is called
  \textit{transversally non-degenerate} if
  $\operatorname{Hess}(\iota^\ast f_2)(m)$ is a non-degenerate
  symmetric bilinear form on ${\rm T}_m S$.
\end{definition}

\begin{remark}
  One can check that Definition~\ref{def:nondegpoi} does not depend on
  the choice of $S$. The existence of $S$ is guaranteed by the
  classical Hamiltonian Flow Box theorem (see e.g., \cite[Theorem
  5.2.19]{AbMa1978}; this result is also called the
  Darboux-Caratheodory theorem~\cite[Theorem 4.1]{PeVN2011}).  It
  guarantees that the condition $\DD_m f_1(m) \neq 0$ ensures the
  existence of a symplectic chart $(x_1,x_2, \xi_1,\xi_2)$ on $M$
  centered at $m$, i.e., $x_i(m)=0, \xi_i(m)=0$, such that
  $\mathcal{H}_{f_1} = \partial/\partial x_1$ and
  $\xi_1=f_1-f_1(m)$. Therefore, since $\ker \left(\DD_m f_1 \right) =
  \operatorname{span} \{\partial/\partial x_1, \partial/\partial x_2,
  \partial/\partial \xi_2\}$, $S$ can be taken to be the local
  embedded symplectic submanifold defined by the coordinates
  $(x_2,\xi_2)$.
\end{remark}

Definition \ref{def:nondegpoi} concerns symplectic four-manifolds,
which is the case relevant to the present paper. For the notion of
non-degeneracy of a critical point in arbitrary dimension see
\cite{vey}, \cite[Section 3]{san-mono}.  Non-degenerate critical
points can be characterized (see \cite{Eliasson1984, Eliasson1990,
  VNWa2010}) using the Williamson normal form \cite{Williamson1936}.
The analytic version of the following theorem by Eliasson is due to
Vey~\cite{vey}.

\begin{figure} [h]
  \centering
  \includegraphics[width=15cm]{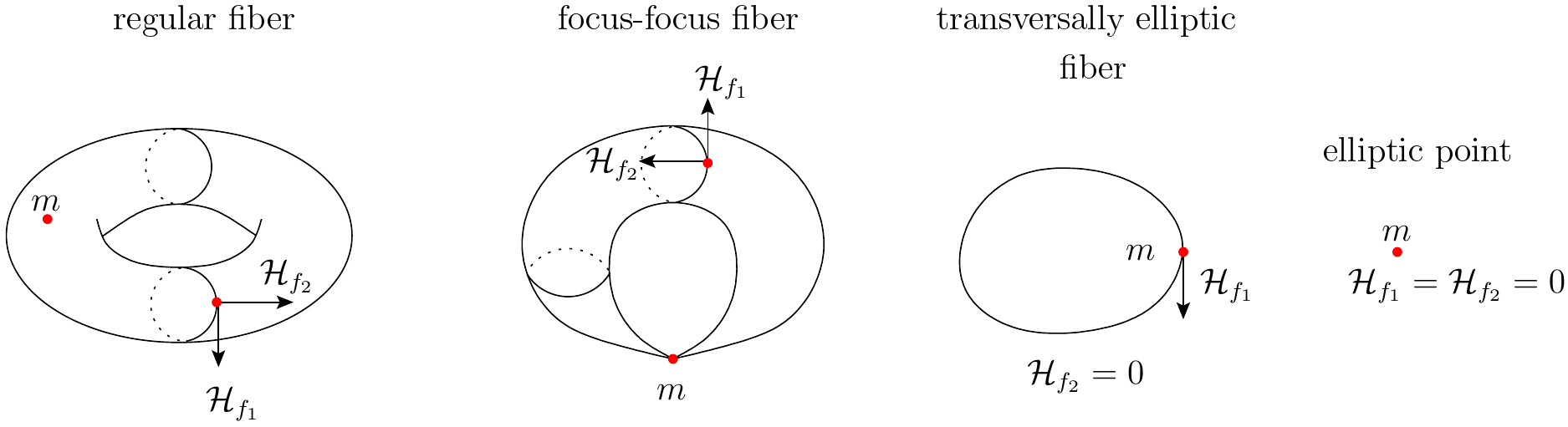}
  \caption{The figures show some possible singularities of a
    integrable system. On the left most figure, $m$ is a regular point
    (rank 2); on the second figure, $m$ is a focus-focus point (rank
    $0$); on the third one, $m$ is a transversally elliptic
    singularity (rank $1$); on the right most figure, $m$ is an
    elliptic-elliptic point (rank 0).}
  \label{singularities}
\end{figure}

\begin{theorem}[H. Eliasson 1990] \label{singularities_theorem} The
  non\--degenerate critical points of a completely integrable system
  $F \colon M \to \mathbb{R}^n$ are linearizable, i.e. if $m \in M$ is
  a non-degenerate critical point of the completely integrable system
  $F=(f_1,\,\ldots,f_n): M \rightarrow \mathbb{R}^n$ then there exist
  local symplectic coordinates $(x_1,\, \ldots,x_n,\, \xi_1,\,
  \ldots,\, \xi_n)$ about $m$, in which $m$ is represented as
  $(0,\,\ldots,\, 0)$, such that $\{f_i,\,q_j\}=0$, for all indices
  $i,\,j$, where we have the following possibilities for the
  components $q_1,\,\ldots,\,q_n$, each of which is defined on a small
  neighborhood of $(0,\,\ldots,\,0)$ in $\mathbb{R}^n$:
  \begin{itemize}
  \item[{\rm (i)}] Elliptic component: $q_j = (x_j^2 + \xi_j^2)/2$,
    where $j$ may take any value $1 \le j \le n$.
  \item[{\rm (ii)}] Hyperbolic component: $q_j = x_j \xi_j$, where $j$
    may take any value $1 \le j \le n$.
  \item[{\rm (iii)}] Focus\--focus component: $q_{j-1}=x_{j-1}\,
    \xi_{j} - x_{j}\, \xi_{j-1}$ and $q_{j} =x_{j-1}\, \xi_{j-1}
    +x_{j}\, \xi_{j}$ where $j$ may take any value $2 \le j \le n-1$
    (note that this component appears as ``pairs'').
  \item[{\rm (iv)}] Non\--singular component: $q_{j} = \xi_{j}$, where
    $j$ may take any value $1 \le j \le n$.
  \end{itemize}
  Moreover if $m$ does not have any hyperbolic component, then the
  system of commuting equations $\{f_i,\,q_j\}=0$, for all indices
  $i,\,j$, may be replaced by the single equation
 $$
 (F-F(m))\circ \varphi = g \circ (q_1,\, \ldots,\,q_n),
 $$ 
 where $\varphi=(x_1,\, \ldots,x_n,\, \xi_1,\, \ldots,\, \xi_n)^{-1}$
 and $g$ is a diffeomorphism from a small neighborhood of the origin
 in $\mathbb{R}^n$ into another such neighborhood, such that $g(0,\,
 \ldots,\,0)=(0,\,\ldots,\,0)$.
\end{theorem}

If the dimension of $M$ is $4$ and $F$ has no hyperbolic singularities
-- which is the case we treat in this paper -- we
have the following possibilities for the map $(q_1,\,q_2)$, depending
on the rank of the critical point:
\begin{itemize}
\item[{\rm (1)}] if $m$ is a critical point of $F$ of rank zero, then
  $q_j$ is one of
  \begin{itemize}
  \item[{\rm (i)}] $q_1 = (x_1^2 + \xi_1^2)/2$ and $q_2 = (x_2^2 +
    \xi_2^2)/2$.
  \item[{\rm (ii)}] $q_1=x_1\xi_2 - x_2\xi_1$ and $q_2 =x_1\xi_1
    +x_2\xi_2$; \,\, on the other hand,
  \end{itemize}
\item[(2)] if $m$ is a critical point of $F$ of rank one, then
  \begin{itemize}
  \item[{\rm (iii)}] $q_1 = (x_1^2 + \xi_1^2)/2$ and $q_2 = \xi_2$.
  \end{itemize}
\end{itemize}
In this case, a non\--degenerate critical point is respectively called
\emph{elliptic\--elliptic, focus\--focus}, or
\emph{transversally\--elliptic} if both components $q_1,\, q_2$ are of
elliptic type, $q_1,\,q_2$ together correspond to a focus\--focus
component, or one component is of elliptic type and the other
component is $\xi_1$ or $\xi_2$, respectively.

Similar definitions hold for \emph{transversally-hyperbolic,
  hyperbolic-elliptic} and \emph{hyperbolic\--hyperbolic} non-degenerate
critical points.

\begin{definition}
  Suppose that $(M, \omega)$ is a connected symplectic
  four-manifold. An integrable system $F \colon M \to \mathbb{R}^2$ is
  called \emph{almost-toric} if all the singularities are
  non-degenerate without hyperbolic components.
\end{definition}

\begin{remark} \label{prop:elliptic} Suppose that $(M, \omega)$ is a
  connected symplectic four-manifold. Let $F \colon M \to
  \mathbb{R}^2$ be an integrable system. If $F$ is a toric integrable
  system, then $F$ is almost-toric, with only elliptic
  singularities. This follows from the fact that a torus action is
  linearizable near a fixed point; see, for instance
  \cite{Delzant1988}.
\end{remark}

A version of the following result is proven in \cite{VN2007} for
almost-toric systems for which the map $F $ is proper. Here we replace
the condition of $F$ being proper by the condition that $F(M)$ is a
closed subset of $\mathbb{R}^2$; this introduces additional
subtleties. Our proof here is independent of the argument in
\cite{VN2007}.

\begin{theorem}
  \label{theo:interior}
  Suppose that $(M, \omega)$ is a connected symplectic four-manifold.
  Assume that $F \colon M \to \R^2$ is an almost\--toric integrable
  system with $B:=F(M)$ closed. Then the set of focus\--focus critical
  values is countable, i.e. we may write it as $\{c_i\mid\, i\in I\}$,
  where $I\subset \NM$. Consider the following statements:
  \begin{itemize}
  \item[{\rm (i)}] the fibers of $F$ are connected;
  \item[{\rm (ii)}] the set $B_r$ or regular values of $F$ is
    connected;
  \item[{\rm (iii)}] for any value $c$ of $F$, for any sufficiently
    small disc $D$ centered at $c$, $B_r\cap D$ is connected;
  \item[{\rm (iv)}] the set of regular values
    is $B_r=\mathring{B}\setminus\{c_i\mid\, i\in I\}$.
    Moreover, the topological boundary $\partial B $ of $B$ consists
    precisely of the values $F(m)$, where $m$ is a critical point of
    elliptic-elliptic or transversally elliptic type.
  \end{itemize}
  Then statement {\rm (i)} implies statement {\rm (ii)}, statement
  {\rm (iii)} implies statement {\rm (iv)}, and statement {\rm (iv)}
  implies statement {\rm (ii)}.

  If in addition $F$ is proper, then statement \textup{(i)} implies
  statement \textup{(iv)}.

\end{theorem}
It is interesting to note that the statement is optimal in that no
other implication is true (except \textup{(iii)}$\impliq$\textup{(ii)}
which is a consequence of the stated implications). This gives an idea
of the various pathologies that can occur for an almost-toric system.
\begin{proof}[Proof of Theorem~\ref{theo:interior}]
  From the local normal form~\ref{singularities_theorem}, focus-focus
  critical points are isolated, and hence the set of focus-focus
  critical points is countable (remember that all our manifolds are
  second countable). Moreover, the image of a focus-focus point is
  necessarily in the interior of $B$.

  Let us show that
  \begin{equation}
    B_r \subset \mathring{B} \setminus \{c_i\mid\, i\in I\}.
    \label{equ:B_r}  
  \end{equation}
  This is equivalent to showing that any value in $\partial B$ is a
  critical value of $F$. Since $B$ is closed, $\partial B \subset B$,
  so for every $c \in \partial B$ we have that $F^{-1}(c)$ is
  nonempty. By the Darboux-Carath\'eodory theorem, the image of a
  regular point must be in the interior of $B$, therefore $F^{-1}(c)$
  cannot contain any regular point: the boundary can contain only
  singular values.

  Since a point in $\partial B$ cannot be the image of a focus-focus
  singularity, it has to be the image of a transversally elliptic or
  an elliptic-elliptic singularity.

  We now prove the implications stated in the theorem.

  \noindent ${\rm (i)}\Rightarrow {\rm (ii)}$: Since $F$ is
  almost-toric, the singular fibers are either points
  (elliptic-elliptic), one-dimensional submanifolds (codimension 1
  elliptic) or a stratified manifold of maximal dimension 2
  (focus-focus and elliptic). This is because none of the critical
  fibers can contain regular tori since the fibers are assumed to be
  connected by hypothesis (i).  The only critical values that can
  appear in one-dimensional families are elliptic and
  elliptic-elliptic critical values (see Figure
  \ref{singularities_in_image}). The focus- focus singularities are
  isolated. Therefore the union of all critical fibers is a locally
  finite union of stratified manifolds of codimension at least 2;
  therefore this union has codimension at least 2. Hence the
  complement is connected and therefore its image by $F$ is also
  connected.

  \noindent ${\rm (iii)} \Rightarrow {\rm (iv)}$: There is no embedded
  line segment of critical values in the interior of $B$ (which would
  come from codimension 1 elliptic singularities) because this is in
  contradiction with the hypothesis of local connectedness
  \textup{(iii)}.  Therefore $\mathring{B} \setminus \{c_i\mid\, i\in
  I\} \subset B_r$.  Hence by~\eqref{equ:B_r}, $$\mathring{B}
  \setminus \{c_i\mid\, i\in I\} = B_r,$$ as desired, and all the
  elliptic critical values must lie in $\partial B$.

  \noindent ${\rm (iv)} \Rightarrow {\rm (ii)}$: As we saw above,
  $F^{-1}(\partial B)$ contains only critical points, of elliptic
  type. Because of the local normal form, the set of rank 1 elliptic
  critical points in $M$ form a 2-dimensional symplectic submanifold
  with boundary, and this boundary is equal to the discrete set of
  rank 0 elliptic points. Therefore $M\setminus F^{-1}(\partial B)$ is
  connected. This set is equal to $F^{-1}(\mathring{B})$, which in
  turn implies that $\mathring{B}$ is connected. By
  hypothesis~\textup{(iv)}, this ensures that $B_r$ is connected.

  Assume for the rest of the proof that $F$ is proper.
 
  \noindent ${\rm (i)}\Rightarrow {\rm (iv)}$: Assume \textup{(iv)}
  does not hold. In view of~\eqref{equ:B_r}, there exists an elliptic
  singularity (of rank 0 or 1) $c$ in the interior of $B$. Let
  $\Lambda$ be the corresponding fiber. Since it is connected, it must
  entirely consist of elliptic points (this comes from the normal form
  Theorem~\ref{singularities_theorem}). The normal form also implies
  that $c$ must be contained in an embedded line segment of elliptic
  singularities, and the points in a open neighborhood $\Omega$ of
  $\Lambda$ are sent by $F$ in only one side of this segment. Since
  $c$ is in the interior of $B$, there is a sequence $c_k\in B$ on the
  other side of the line segment that converges to $c$ as
  $k\to\infty$. Hence there is a sequence $m_k\in M\setminus\Omega$
  such that $F(m_k)=c_k$. Since $F$ is proper, one can assume that
  $m_k$ converges to a point $m$ (necessarily in
  $M\setminus\Omega$). By continuity of $F$, $m$ belongs to the fiber
  over $c$, and thus to $\Omega$, which is a contradiction.

\end{proof}

\begin{figure}[h]
  \centering
  \includegraphics[width=12.5cm]{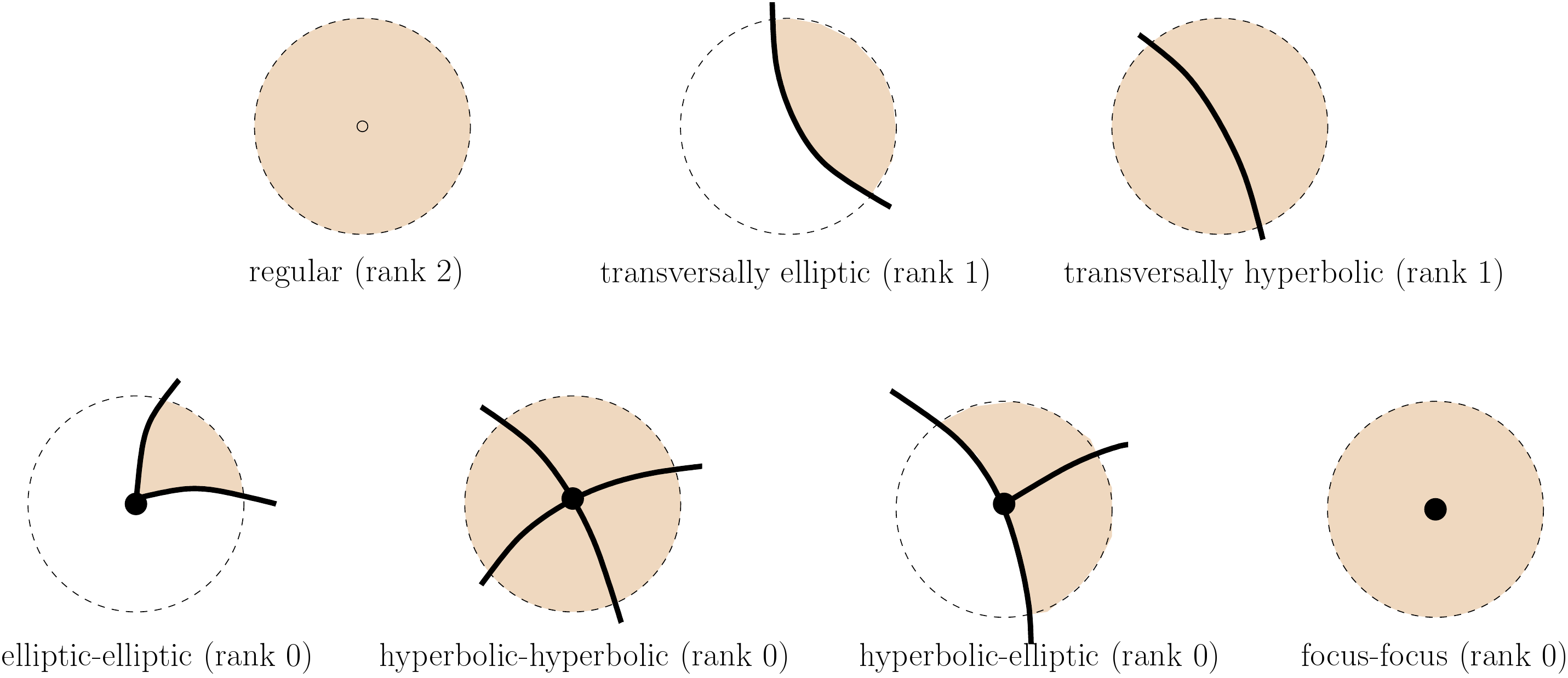}
  \caption{The local classification of the range of $F$ in an open
    disk.}
  \label{singularities_in_image}
\end{figure}

\section{The fibers of an almost\--toric system} \label{fibers:sec}

In this section we study the structure of the fibers of an
almost-toric system.

We shall need below the definition and basic properties of the
bifurcation set of a smooth map.
 
\begin{definition} \label{def:bifurcation} Let $M$ and $N$ be smooth
  manifolds. A smooth map $f: M\rightarrow N$ is said to be
  \emph{locally trivial} at $n_0 \in f(M)$ if there is an open
  neighborhood $U \subset N$ of $n_0$ such that $f ^{-1}(n)$ is a
  smooth submanifold of $M$ for each $n\in U$ and there is a smooth
  map $h: f^{-1}(U)\rightarrow f ^{-1}(y_0)$ such that $f\times
  h:f^{-1}(U)\rightarrow U \times f ^{-1}(n_0)$ is a
  diffeomorphism. The \emph{bifurcation set} $\Sigma_f$ consists of
  all the points of $N$ where $f$ is not locally trivial.
\end{definition}

Note, in particular, that $h|_{f ^{-1}(n)}:f ^{-1}(n) \rightarrow
f^{-1}(n_0)$ is a diffeomorphism for every $n \in U $. Also, the set
of points where $f $ is locally trivial is open in $N$.

\begin{remark}
  Recall that $\Sigma_f $ is a closed subset of $N $.  It is well
  known that the set of critical values of $f $ is included in the
  bifurcation set (see \cite[Proposition 4.5.1]{AbMa1978}).  In
  general, the bifurcation set strictly includes the set of critical
  values. This is the case for the momentum-energy map for the
  two-body problem \cite[\S9.8]{AbMa1978}. However (see \cite[Page
  340]{AbMa1978}), if $f \colon M \to N$ is a smooth \emph{proper}
  map, then the bifurcation set of $f$ is equal to the set of critical
  values of $f$.
\end{remark}

 Recall that a smooth map
$f:M \rightarrow \mathbb{R}$ is \emph{Morse} if all its critical
points are non-degenerate.  The smooth map $f$ is \emph{Morse-Bott} if
the critical set of $f$ is a disjoint union of connected submanifolds
$\mathcal{C}_i$ of $M$, on which the Hessian of $f$ is non-degenerate
in the transverse direction, i.e.,
$$
\ker (\operatorname{Hess}_m f) = {\rm T}_m \mathcal{C}_i,\,\,\,
\textup{for all}\,\,i, \,\,\textup{for all} \,\, m \in \mathcal{C}_i.
$$
The \emph{index of} $m$ is the number of negative eigenvalues of
$(\operatorname{Hess}f)(m)$.  

The goal of this section is to find a useful Morse theoretic result,
valid in great generality and interesting on its own, that will
ultimately imply the connectedness of the fibers of an integrable
system (see Theorem \ref{theo:fibers}).  Here we do not rely on
Fomenko's Morse theory~\cite{Fomenko1986}, because we do not want to
select a nonsingular energy surface. Instead, the model is \cite[Lemma
5.51]{McSa1994}; however, the proof given there does not extend to the
non-compact case, as far as we can tell. We thank Helmut Hofer and Thomas
Baird for sharing their insights on Morse theory with us that helped
us in the proof of the following result.

\begin{lemma}
  \label{lemm:semibounded}
  Let $f:M\fleche\RM$ be a Morse-Bott function on a connected manifold
  $M$. Assume $f$ is proper and bounded from below and has no critical
  manifold of index 1. Then the set of critical points of index 0 is
  connected.
\end{lemma}

\begin{proof}
  We endow $M$ with a Riemannian metric.  The negative gradient flow
  of $f$ is complete. Indeed, along the the flow the function $f$
  cannot increase and, by hypothesis, $f$ is bounded from
  below. Therefore, the values of $f$ remain bounded along the
  flow. By properness of $f$, the flow remains in a compact subset of
  $M$ and hence it is complete.

  Let us show, using standard Morse-Bott theory, that the integral
  curve of $- \nabla f$ starting at any point $m\in M$ tends to a
  critical manifold of $f$. In the compact set in $\{x \in M \mid
  f(x)\leq f(m)\}$, there must be a finite number of critical
  manifolds contained . If the integral curve avoids a neighborhood of
  these critical manifolds, by compactness it has a limit point, and
  by continuity the vector field at the limit point must vanish; we
  get a contradiction, thus proving the claim.

  Thus, we have the disjoint union $M = \bigsqcup_{k=0}^n W^s(C_k)$,
  where $C_k$ is the set of critical points of index $k$, and
  $W^s(C_k)$ is its stable manifold:
  \[
  W^s(C_k) := \{m\in M\mid d(\phy_{-\nabla f}^t(m),C_k)\fleche 0
  \text{ as } t\fleche+\infty\},
  \]
  where $d$ is any distance compatible with the topology of $M$ (for
  example, the one induced by the given Riemannian metric on $M$) and
  $t \mapsto \phy_{-\nabla f}^t$ is the flow of the vector field $-
  \nabla f$. Since $f$ has no critical point of index 1, we have
  \[
  C_0 = W^s(C_0) = M \setminus \bigsqcup_{k=2}^n W^s(C_k).
  \]
  The local structure of Morse-Bott singularities given by the
  Morse-Bott lemma~\cite{Bott,Banyaga}) implies that $W^s(C_k)$ is a
  submanifold of codimension $k$ in $M$.  Hence $\bigsqcup_{k=2}^n
  W^s(C_k)$ cannot disconnect $M$.
\end{proof}

\begin{remark}
  Since all local minima of $f$ are in $C_0$, we see that $C_0$ must
  be the set of global minima of $f$; thus $C_0$ must be equal to the
  level set $f^{-1}(f(C_0))$.
\end{remark}

\begin{prop} \label{prop:levelset}

  Let $M$ be a connected smooth manifold and $f:M\fleche\RM$ be a
  proper Morse-Bott function whose indices and co-indices are always
  different from 1. Then the level sets of $f$ are connected.
\end{prop}

\begin{proof}
  Let $c$ be a regular value of $f$ (such a value exists by Sard's
  theorem). Then $g:=(f-c)^2$ is a Morse-Bott function.  On the set
  $\{x \in M \mid f(x)>c\}$, the critical points of $g$ coincide with
  the critical points of $f$ and they have the same index. On the set
  $\{x \in M \mid f(x)<c\}$, the critical points of $g$ also coincide
  with the critical points of $f$ and they have the same coindex. The
  level set $\{x \in M \mid f(x)=c\}$ is clearly a set of critical
  points of index 0 of $g$. Of course, $g$ is bounded from
  below. Thus, by Lemma \ref{lemm:semibounded}, the set of critical
  points of index 0 of $g$ is connected (it may be empty) and hence
  equal to $g^{-1}(0)$. Therefore $f^{-1}(c)$ is connected.  This
  shows that all regular level sets of $f$ are connected.  (As usual,
  a regular level set --- or regular fiber --- is a level set that
  contains only regular points, i.e. the preimage of a regular value.)

  Finally let $c_i$ be a critical value of $f$ (if any). Since $f$ is
  proper and has isolated critical manifolds, the set of critical
  values is discrete. Let $\epsilon_0>0$ such that the interval
  $[c_i-\epsilon_0,c_i+\epsilon_0]$ does not contain any other
  critical value.  Consider the manifold $N:=M\times S^2$, and, for
  any $\epsilon\in(0,\epsilon_0)$, let
  \[
  h_\epsilon:=f-c_i+\epsilon z:N\fleche\RM,
  \]
  where $z$ is the vertical component on the sphere
  $S^2\subset\RM^3$. Notice that $z:S^2\fleche\RM$ is a Morse function
  with indices 0 and 2. Thus $h_\epsilon$ is a Morse-Bott function on
  $N$ with indices and coindices of the same parity as those of
  $f$. Thus no index nor coindex of $h$ can be equal to 1. By the
  first part of the proof, the regular level sets of $h_\epsilon$ must
  be connected. The definition of $\epsilon_0$ implies that $0$ is a
  regular value of $h_\epsilon$. Thus
  \[
  F_\epsilon:=\pi_M(h_\epsilon^{-1}(0)) = \{m\in M\mid \quad
  \abs{f(m)-c_i}\leq \epsilon\}
  \]
  is connected.  Since $f$ is proper, $F_\epsilon$ is also
  compact. Because a non-increasing intersection of compact connected
  sets is connected, we see that
  $f^{-1}(c_i)=\cap_{0<\epsilon\leq\epsilon_0}F_\epsilon$ is
  connected.
\end{proof}

There is no a priori reason why the fibers $F^{-1}(x,\,y)= J^{-1}(x)
\cap H^{-1}(y)$ of $F$ should be connected even if $J$ and $H$ have
connected fibers (let alone if just one of $J$ or $H$ has connected
fibers). However, the following result shows that this conclusion is
holds. To prove it, we need a preparatory lemma which is interesting
on its own.

\begin{lemma} \label{tricky} Let $f:X \to \mathbb{R}^n$ be a map from
  a smooth connected manifold $X$ to $\mathbb{R}^n$. Ler $B_r$ be the
  set of regular values of $f$. Suppose that $f$ has the following
  properties.
  \begin{itemize}
  \item[{\rm (1)}] $f$ is a proper map.
  \item[{\rm (2)}] For every sufficiently small neighborhood $D$ of
    any critical value of $f$, $B_r \cap D$ is connected.
  \item[{\rm (3)}] The regular fibers of $f$ are connected.
  \item[{\rm (4)}] The set $\mathrm{Crit}(f)$ of critical points of
    $f$ has empty interior.
  \end{itemize}
  Then the fibers of $f$ are connected.
\end{lemma}

\begin{proof}
  We use the following ``fiber continuity'' fact~: that if $\Omega$ is
  a neighborhood of a fiber $f^{-1}(c)$ of a continuous proper map
  $f$, then the fibers $f^{-1}(q)$ with $q$ close to $c$ also lie
  inside $\Omega$.  Indeed, if this statement were not true, then
  there would exist a sequence $q_n\rightarrow c$ and a sequence of
  points $x_n \in f^{-1}(q_n)$, $x_{n} \notin \Omega$, such that there
  is a subsequence $x_{n_k} \rightarrow x \notin \Omega$. However, by
  continuity $x\in f^{-1}(c)$ which is a contradiction.

  Assume a fiber $\mathcal{F}=f^{-1}(p)$ of $f$ is not connected. Then
  there are disjoint open sets $U$ and $V$ in $X$ such that
  $\mathcal{F}$ lies in $U \cup V$ but is not contained in either $U$
  or $V$.

  By fiber continuity, there exists a small open disk $D$ about $p$
  such that $f^{-1}(D)\subset U\cup V$.

  Since the regular fibers are connected, we can define a map $\psi:
  D\cap B_r\to\{0;1\}$ which for $c\in D\cap B_r$ is equal to $1$ if
  $f^{-1}(c)\subset U$, and is equal to $0$ if $f^{-1}(c)\subset
  V$. The fiber continuity says that the sets $\psi^{-1}(0)$ and
  $\psi^{-1}(1)$ are open, thus proving that $\psi$ is
  continuous. By~(2), the image of $\psi$ must be connected, and
  therefore $\psi$ is constant. We can hence assume without loss of
  generality that all regular fibers above $D$ are contained in $U$.

  Now consider the restriction $\tilde{f}$ of $f$ on the open set
  $V\cap f^{-1}(D)$. Because of the above argument, it cannot take any
  value in $B_r\cap D$. Thus this map takes values in the set of
  critical values of $f$, which has measure zero by Sard's
  theorem. This requires that on $V\cap f^{-1}(D)$, the rank of $\DD
  f$ be strictly less that $n$, which contradicts~(4), and hence
  proves the lemma~: $f^{-1}(p)$ has to be connected.
\end{proof}

Now we are ready to prove one of our main results.

\begin{theorem} \label{theo:fibers} Suppose that $(M, \omega)$ is a
  connected symplectic four-manifold. Let $F=(J,\,H) \colon M \to
  \mathbb{R}^2$ be an almost\--toric integrable system such that $F$
  is a proper map. Suppose that $J$ has connected fibers, or that $H$
  has connected fibers.  Then the fibers of $F$ are connected.
\end{theorem}

\begin{proof}
  Without loss of generality, we may assume that $J$ has connected
  fibers.

  \paragraph{Step 1.} We shall prove first that for every
  \textit{regular value} $(x,\,y)$ of $F$, the fiber $F^{-1}(x,\,y)$
  is connected. To do this, we divide the proof into two cases.
\\
\\
  \emph{Case 1A.} Assume $x$ is a regular value of $J$.  Then the
  fiber $J^{-1}(x)$ is a smooth manifold. Let us show first that the
  non-degeneracy of the critical points of $F $ and the definition of
  almost-toric systems implies that the function $H_x:=H|_{J^{-1}(x)}
  \colon J^{-1}(x) \to \mathbb{R} $ is Morse-Bott. Let $B_r$ be the
  set of regular values of $F$.

  Let $m_0$ be a critical point of $H_x$. Then there exists
  $\lambda\in \mathbb{R}$ such that $\DD H(m_0)=\lambda \DD
  J(m_0)$. Thus $m_0$ is a critical point of $F$; it must be of rank
  $1$ since $\DD J$ never vanishes on $J^{-1}(x)$. Since $F$ is an
  almost-toric system, the only possible rank $1$ singularities are
  transversally elliptic singularities, i.e., singularities with 
  one elliptic component and one non singular component in Theorem
  \ref{singularities_theorem}; see Figure
  \ref{singularities_in_image}. Thus, by Theorem
  \ref{singularities_theorem}, there exist local canonical coordinates
  $(x_1,\,x_2,\,\xi_1,\,\xi_2)$ such that $F=g(x_1^2+\xi_1^2,\,\xi_2)$
  for some local diffeomorphism $g$ of $\mathbb{R}^2$ about the origin
  and fixing the origin; thus the derivative
  \[
  {\rm D}g(0,0)=:
  \begin{pmatrix}
    a & b\\c & d
  \end{pmatrix} \in \operatorname{GL}(2, \mathbb{R}).
  \]
  Note $\DD J(m_0)\neq 0$ implies that $d\neq 0$. Therefore, by the
  implicit function theorem, the submanifold $J^{-1}(x)$ is locally
  parametrized by the variables $(x_1,x_2,\xi_1)$ and, within it, the
  critical set of $H_x$ is given by the equation $x_1=\xi_1=0$; this
  is a submanifold of dimension 1. The Taylor expansion of $H_x$ is
  easily computed to be
  \begin{equation}
    H_x = a(x_1^2+\xi_1^2) -\frac{bc}{d}(x_1^2+\xi_1^2) 
    + \frac{b}{d}x + \mathcal{O}(x_1,\xi_1,x_2)^3.
    \label{equ:Hx}  
  \end{equation}
  Thus, the coefficient of $(x_1^2+\xi_1^2)$ is $(a-\frac{bc}{d})$
  which is non-zero and hence the Hessian of $H_x$ is transversally
  non-degenerate. This proves that $H_x$ is Morse-Bott, as claimed.

  Second, we prove, in this case, that the fibers of $F$ are
  connected. At $m_0$, the transversal Hessian of $H_x$ has either no
  or two negative eigenvalues, depending on the sign of
  $(a-\frac{bc}{d})$. This implies that each critical manifold has
  index $0$ or index $2$. If this coefficient is negative, the sum of
  the two corresponding eigenspaces is the full $2$-dimensional
  $(x_1,\xi_1)$-space.
  
  Note that $H_x \colon J^{-1}(x) \to \mathbb{R}$ is a proper map:
  indeed, if $K \subset \mathbb{R}$ is compact, then
  $H_x^{-1}(K)=F^{-1}(\{x\} \times K)$ is compact because $F$ is
  proper. Thus $H_x \colon J^{-1}(x) \to \mathbb{R}$ is a smooth
  Morse-Bott function on the connected manifold $J^{-1}(x)$ and $H_x$
  and $-H_x$ have only critical points of index $0$ or $2$.  We are in
  the hypothesis of Proposition~\ref{prop:levelset} and so we can
  conclude that the fibers of $H_x \colon J^{-1}(x) \to \mathbb{R}$
  are connected.

  Now, since $(H_x)^{-1}(y)=F^{-1}(x,\,y)$, it follows that
  $F^{-1}(x,\,y)$ is connected for all $(x,\,y) \in F(M) \subset
  \mathbb{R}^2$ whenever $x$ is a regular value of $J$.
 \\
 \\
  \emph{Case 1B.} Assume that $x $ is not a regular value of
  $J$. Note that there exists a point $(a,\,b)$ in every connected
  component $C_r$ of $B_r$ such that $b$ is a regular value for $J$;
  otherwise $\DD J$ would vanish on $F^{-1}(C_r)$, which violates the
  definition of $B_r$.  The restriction $F|_{F^{-1}(C_r)} \colon
  F^{-1}(C_r) \to C_r$ is a locally trivial fibration since, by
  assumption, $F$ is proper and thus the bifurcation set is equal to
  the critical set. Thus all fibers of $F|_{F^{-1}(C_r)}$ are
  diffeomorphic. It follows that $F^{-1}(x,\,y)$ is connected for all
  $(x,\,y) \in C_r$.

  This shows that all inverse images of regular values of $F$ are
  connected.

  \paragraph{Step 2.} We need to show that $F ^{-1}(x,y)$ is
  connected if $(x,y)$ is not a regular value of $F$.  We claim that
  there is \emph{no} critical value of $F$ in the interior of the
  image $F(M)$, except for the critical values that are images of
  focus-focus critical points of $F$. Indeed, if there was such a
  critical value $(x_0,\,y_0)$, then there must exist a small segment
  line $\ell$ of critical values (by the local normal form described
  in Theorem \ref{singularities_theorem} and Figure
  \ref{singularities_in_image}). Now we distinguish two cases.
\\
\\
  \emph{Case 2A.} First assume that $\ell$ is not a vertical
  segment (i.e., contained in a line of the form
  $x=\textup{constant}$) and let $\hat{\ell}:=F^{-1}(\ell)$.  Then
  $J(\hat{\ell})$ contains a small interval around $x$, so by Sard's
  theorem, it must contain a regular value $x_0$ for the map $J$. Then
  $J^{-1}(x_0)$ is a smooth manifold which is connected, by
  hypothesis. By the argument earlier in the proof (see Step 1, Case
  A), both $H$ and $-H$ restricted to $J^{-1}(x_0)$ are proper
  Morse-Bott functions with indices $0$ and $2$. So, if there is a
  local maximum or local minimum, it must be unique. However, the
  existence of this line of rank 1 elliptic singularities implies that
  there is a local maximum/minimum of $H_{x_0}$ (see
  formula~\eqref{equ:Hx}).  Since the corresponding critical value
  lies in the interior of the image of $H_{x_0}$, it cannot be a
  global extremum; we arrived at a contradiction. Thus the small line
  segment $\ell$ must be vertical.
  \\
  \\
  \emph{Case 2B.} Second, suppose that $\ell$ is a vertical
  segment (i.e., contained in a line of the form
  $x=\textup{constant}$) and let $\hat{\ell}:=F^{-1}(\ell)$.  We can
  assume, without loss of generality, that the connected component of
  $\ell$ in the bifurcation set is vertical in the interior of $F(M)$;
  indeed, if not, apply Case 2A a above.
  
  From Figure~\ref{singularities_in_image} we see that $\hat{\ell}$
  must contain at least one critical point $A$ of transversally
  elliptic type together with another point $B$ either regular or of
  transversally elliptic type. By the normal form of non-degenerate
  singularities, $J^{-1}(x_0)$ must be locally path connected. Since
  it is connected by assumption, it must be path connected. So we have
  a path $\gamma: [0,1]\mapsto J^{-1}(x_0)$ such that $\gamma(0)=A$
  and $\gamma(1)=B$. Near $A$ we have canonical coordinates
  $(x_1,x_2,\xi_1,\xi_2)\in\RM^4$ and a local diffeomorphism $g$
  defined in a neighborhood of the origin of $\mathbb{R}^2$ and
  preserving it, such that
$$
F=g(x_1^2+\xi_1^2,\xi_2)
$$ 
and $A=(0,0,0,0)$. Write $g=(g_1,g_2)$. The critical set is defined by
the equations $x_1=\xi_1=0$ and, by assumption, is mapped by $F$ to a
vertical line. Hence $g_1(0,\xi_2)$ is constant, so
$$
\partial_2g_1(0,\xi_2)=0.
$$ 
Since $g$ is a local diffeomorphism, $\DD
g_1\neq 0$, so we must have $\partial_1 g_1 \neq 0$. Thus, by the
implicit function theorem, any path starting at $A$ and satisfying
$g_1(x_1^2\xi_1^2,\xi_2) =\text{constant}$ must also satisfy
$x_1^2+\xi_1^2=0$.  Therefore, $\gamma$ has to stay in the critical
set $x_1=\xi_1=0$.

Assume first that $\gamma([0,1])$ does not touch the boundary of
$F(M)$. Then this argument shows that the set of $t\in[0,1]$ such that
$\gamma(t)$ belongs to the critical set of $F$ is open. It is also
closed by continuity of $\DD F$.  Hence it is equal to the whole
interval $[0,1]$. Thus $B$ must be in the critical set; this rules out
the possibility for $B$ to be regular. Thus $B$ must be a rank-1
elliptic singularity.  Notice that the sign of $\partial_1 g_1$
indicates on which side of $\ell$ (left or right) lie the values of
$F$ near $A$.
  
Thus, even if $g_1$ itself is not globally defined along the path
$\gamma$, this sign is locally constant and thus globally defined
along $\gamma$. Therefore, all points near $\hat{\ell}$ are mapped by
$F$ to the same side of $\ell$, which says that $\ell$ belongs to the
boundary of $F(M)$; this is a contradiction.
  
Finally, assume that $\gamma([0,1])$ touches the boundary of
$F(M)$. From the normal form theorem, this can only happen when the
fiber over the contact point contains an elliptic-elliptic point
$C$. Thus there are local canonical coordinates
$(x_1,x_2,\xi_1,\xi_2)\in\mathbb{R}^4$ and a local diffeomorphism $g$
defined in a neighborhood of the origin of $\mathbb{R}^2$ and
preserving it, such that
$$
F=g(x_1^2+\xi_1^2,x_2^2+\xi_2^2)
$$ 
near $C=(0,0,0,0)$. We note that the same argument as above applies:
simply replace the $\xi_2$ component by $x_2^2+\xi_2^2$. Thus we get
another contradiction.  Therefore there are no \emph{critical} values
$c$ in the interior of the image $F(M)$ other than focus-focus values
(i.e., images of focus-focus points).

\paragraph{Step 3.}
We claim here that for any critical value $c$ of $F$ and for any
sufficiently small disk $D$ centered at $c$, $B_r\cap D$ is connected.

First we remark that Step 2 implies that item \textup{(iv)} in
Theorem~\ref{theo:interior} holds, and hence item \textup{(ii)} must
hold~: the set of regular values of $F$ is connected.

If $c$ is a focus-focus value, it must be contained in the interior of
$F(M)$, therefore it follows from Step 2 that it is isolated~: there
exists a neighborhood of $c$ in which $c$ is the only critical value,
which proves the claim in this case.

We assume in the rest of the proof that $c=(x_c,y_c)$ is an elliptic
(of rank 0 or 1) critical value of $F$. Since we have just proved in
Step 2 that there are no critical values in the interior of $F(M)$
other than focus-focus values, we conclude that $c \in \partial
(F(M))$. Moreover, the fiber cannot contain a regular Liouville
torus. Then, again by Theorem \ref{singularities_theorem}, the only
possibilities for a neighborhood of $c$ in $F(M)$ are superpositions
of elliptic local normal forms of rank 0 or 1 (given by Theorem
\ref{singularities_theorem}) in such a way that $c \in \partial
(F(M))$.  If only one local model appears, then the claim is
immediate.

Let us show that a neighborhood $U$ of $c$ cannot contain several
different images of local models. Indeed, consider the possible
configurations for two different local images $C_1$ and $C_2$: either
both $C_1$ and $C_2$ are elliptic-elliptic images, or both are
transversally elliptic images, or $C_1$ is an elliptic-elliptic image
and $C_2$ is a transversally elliptic image. Step 2 implies that the
critical values of $F$ in $C_1$ and $C_2$ can only intersect at a
point, provided the neighborhood $U$ is taken to be small enough.  Let
us consider a vertical line $\ell$ through $C_1$ which corresponds to
a regular value of $J$. Any crossing of $\ell$ with a non-vertical
boundary of $F(M)$ must correspond to a local extremum of
$H|_{F^{-1}(\ell)}$, and by Step 2 this local extremum has to be a
global one. Since only one global maximum and one global minimum are
possible, the only allowed configurations for $C_1$ and $C_2$ are such
that the vertical line $\ell_c$ through $c$ separates the regular
values of $C_1$ from the regular values of $C_2$ (see
Figure~\ref{fig:C1C2}).
\begin{figure}[h]
  \centering
  \includegraphics[width=0.45\textwidth]{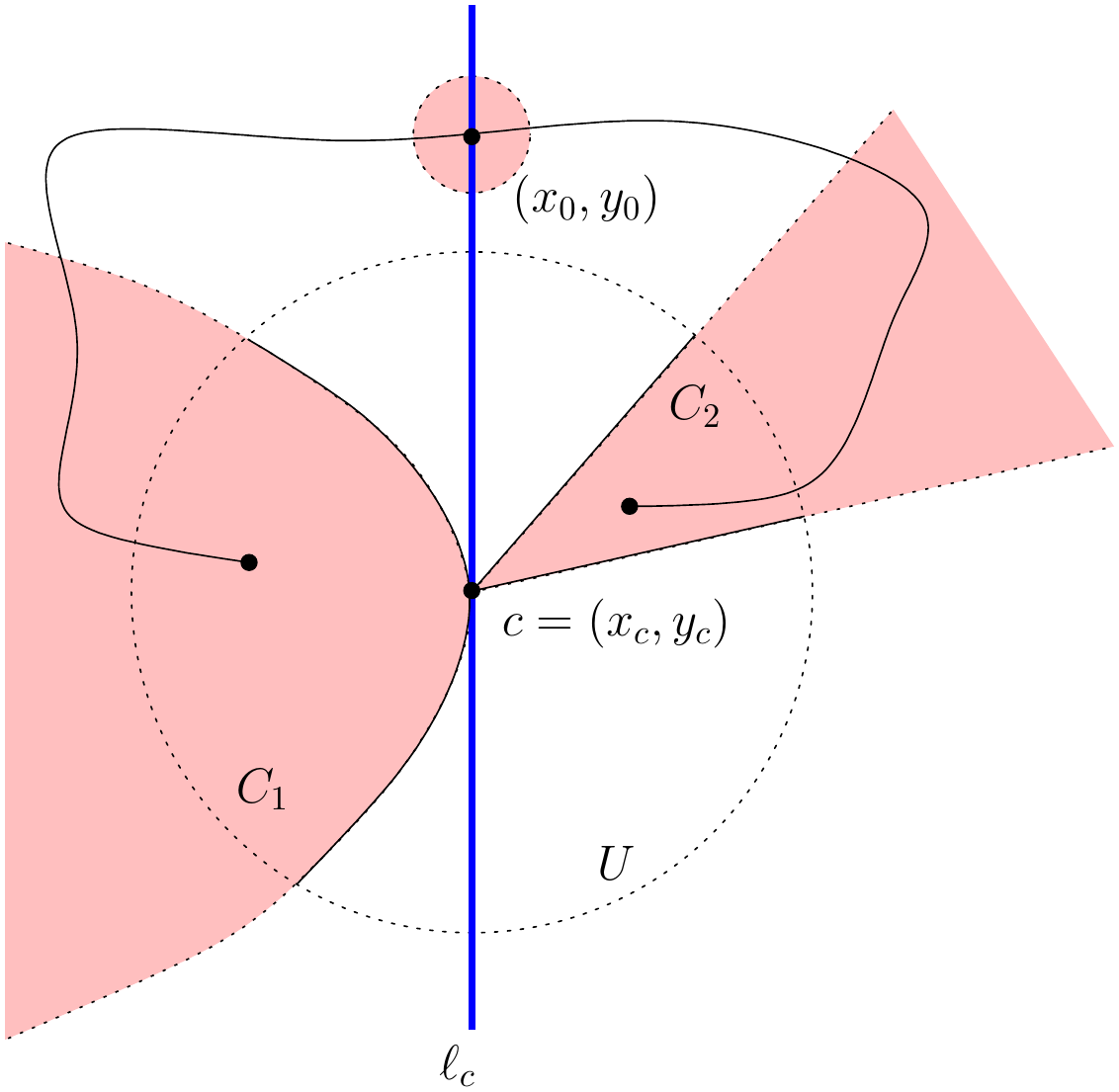}
  \caption{Point overlapping of images near singularities. Such a
    situation cannot occur if the fibers of $J$ are connected.}
  \label{fig:C1C2}
\end{figure}

Since $B_r$ is path connected, there exists a path in $B_r$ connecting
a point in $B_r\cap C_1$ to a point in $B_r\cap C_2$. By continuity
this path needs to cross $\ell_c$, and the intersection point
$(x_0=x_c,y_0)$ must lie outside $U$. Therefore there exists an open
ball $B_0\subset B_r$ centered at $(x_0,y_0)$.  Suppose for instance
that $y_0>y_c$. Then for $(x,y)\in U$, $y$ cannot be a maximal value
of $H|_{J^{-1}(x)}$, which means that for each of $C_1$ and $C_2$,
only local minima for $H|_{J^{-1}(x)}$ are allowed. This cannot be
achieved by any of the local models, thus finishing the proof of our
claim.

The statement of the theorem now follows from Lemma \ref{tricky} since
all fibers are codimension at least one (this follows directly from
the independence assumption in the definition of an integrable
system. But, in fact, we know from the topology of non-degenerate
integrable systems that fibers have codimension at least
two~\cite{bolsinov-fomenko-book}).
\end{proof}

\begin{remark}
  It is not true that an almost-toric integrable system with connected
  regular fibers has also connected singular fibers. See
  example~\ref{ex0} below.
\end{remark}

The following are examples of almost-toric systems in which the fibers
of $F$ are not connected. In the next section we will combine
Theorem~\ref{theo:fibers} with an upcoming result on contact theory
for singularities (which we will prove too) in order to obtain
Theorem~\ref{theo:main-connectivity} of Section~\ref{sec:intro}.

\begin{example} \label{ex0} This example appeared in \cite[Chapter 5,
  Figure 29]{san-mono}. It is an example of a toric system $F:=(J,\,H)
  \colon M\to \mathbb{R}^2$ on a compact manifold for which $J$ and
  $H$ have some disconnected fibers (the number of connected
  components of the fibers also changes). Because this example is
  constructed from the standard toric system $S^2 \times S^2$ by
  precomposing with a local diffeomorphism, the singularities are
  non\--degenerate. In this case the fundamental group $\pi_1(F(M))$
  has one generator, so $F(M)$ is not simply connected, and hence not
  contractible.  See Figure~\ref{anneau}. An extreme case of this
  example can be obtained by letting only two corners overlap
  (Figure~\ref{fig:anneau-extreme}). We get then an almost-toric
  system where all regular fibers are connected, but one singular
  fiber is not connected.
\end{example}

\begin{figure}[htbp]
  \centering
  \includegraphics[width=11.5cm]{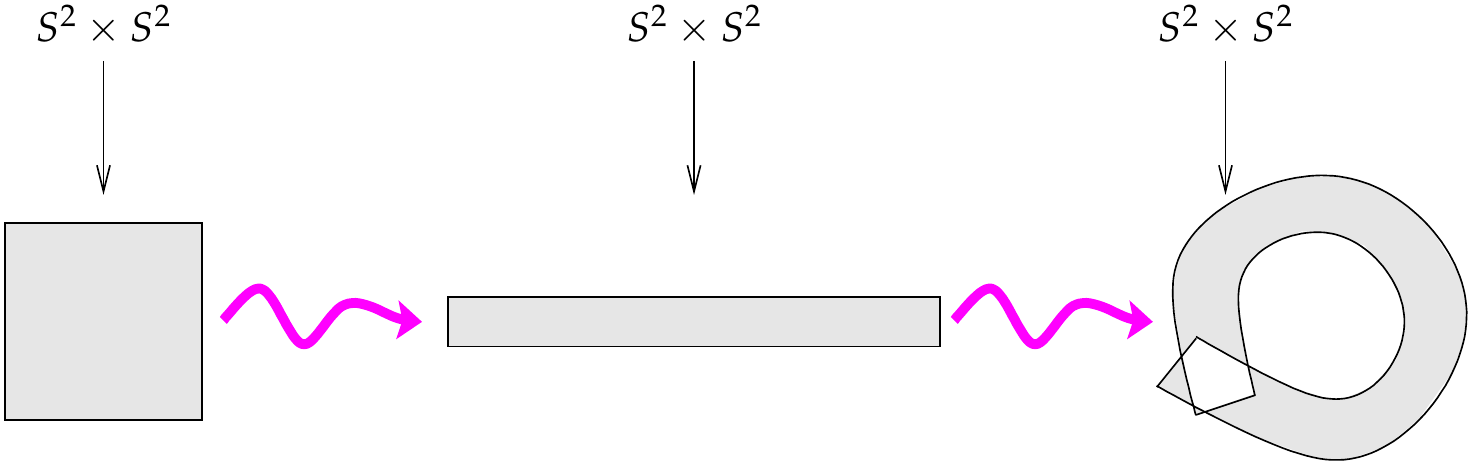}
  \caption{An almost-toric system with some disconnected fibers,
    constructed from the standard toric system on $S^2 \times S^2$ by
    precomposing with a local diffeomorphism.}
  \label{anneau}
\end{figure}

\begin{figure}[h]
  \centering
  \includegraphics[width=2.5cm]{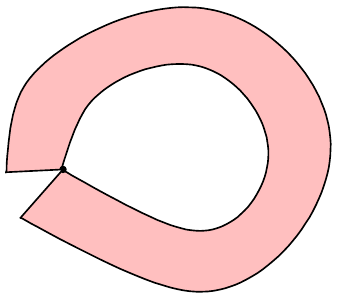}
  \caption{Image of integrable system constructed from $S^2 \times
    S^2$ by a one-point identification, and which has one disconnected
    fiber. All the other fibers are connected.}
  \label{fig:anneau-extreme}
\end{figure}

\begin{example} \label{ex2} The manifold is $M := S^2 \times S^1
  \times S^1$.  Choose coordinates $h \in [1,\,2]$, $a \in
  \mathbb{R}/2 \pi \mathbb{Z}$ on $S^2$; these are action-angle
  coordinates for the Hamiltonian system given by the rotation
  action. Choose coordinates $b, c \in \mathbb{R}/2\pi\mathbb{Z}$ on
  $S^1 \times S^1$. For each $n \in \mathbb{N}$, $n \neq0$, the 2-form
  $\DD h\wedge\DD a + n\DD b\wedge\DD c$ is symplectic. Let
$$
F(h, \,a,\,b,\,c) := ( J:=h \cos( n b) , \, H:=h \sin( n b) )
$$  
and note that $J$ is the momentum map of an $S^1$-action rotating the
sphere about the vertical axis and the first component of $S ^1\times
S ^1$.

Note that $F$ maps $M$ onto the annulus in Figure \ref{fig:annulus}.
\begin{figure}[h]
  \centering
  \includegraphics[width=0.5\textwidth]{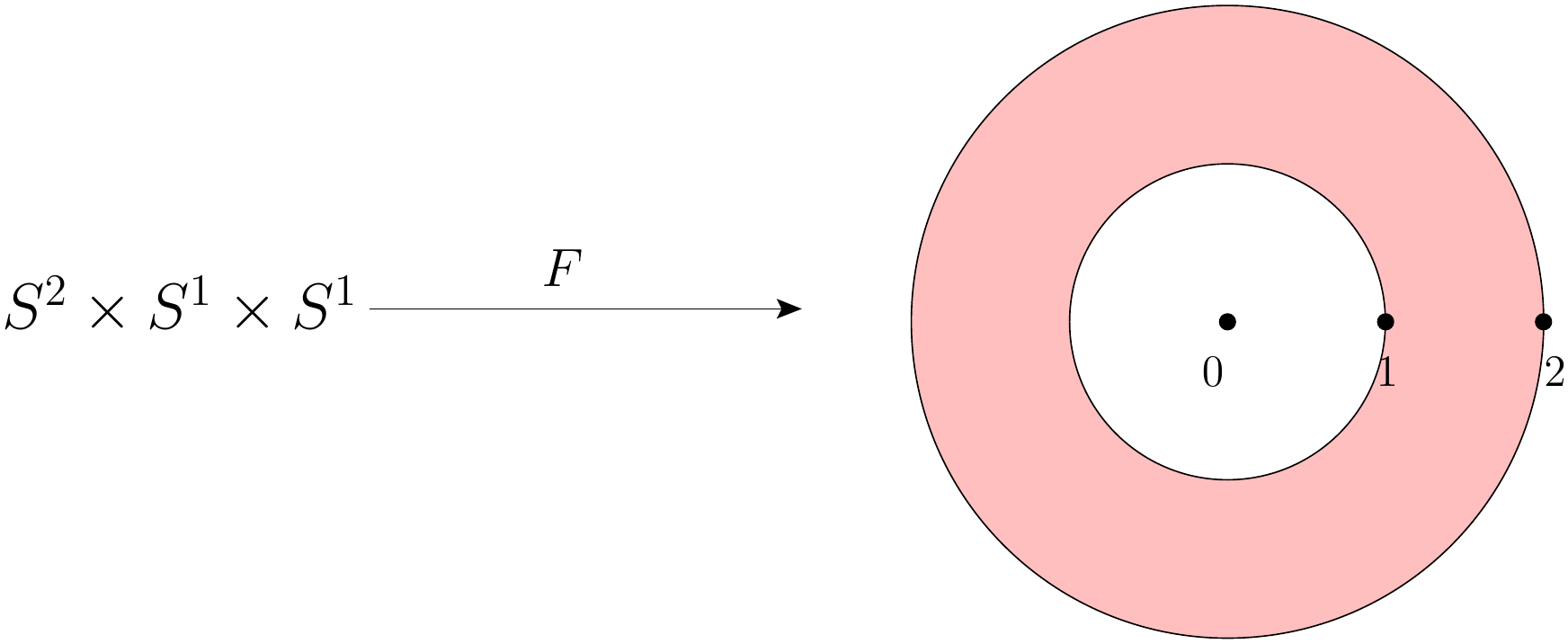}
  \caption{An almost toric system with disconnected fibers.}
  \label{fig:annulus}
\end{figure}
Topologically $F$ winds the first copy of $S^1$ exactly $n$ times
around the annulus and maps $S^2$ to radial intervals.  The fiber of
any regular value is thus $n$ copies of $(S^1)^2$, while the preimage
of any singular value is $S^1$.  One can easily check that all
singular fibers are transversely elliptic. The Hamiltonian vector
fields are:
$$
\mathcal{H}_J=\cos(n b)\frac{\partial}{\partial a} - h \sin(
nb)\frac{\partial}{\partial c} ,\,\,\,\,\,\, \mathcal{H}_H=\sin(n b)
\frac{\partial}{\partial a} + h \cos(nb) \frac{\partial}{\partial c}
$$
which are easily checked to commute (the coefficient functions are
invariants of the flow). Observe that neither integrates to a global
circle action. As for the components, they all look alike (the system
has rotational symmetry in the $b$ coordinate).  Consider for instance
the component $J(h, a, b, c) =h \cos(nb)$ which has critical values
$-2, -1, 1, 2$ and is Morse-Bott with critical sets equal to $n$
copies of $S^1$ and with Morse indices $0, 2, 1$, and $3$,
respectively. The sets $J < -2$ and $J > 2$ are empty.  For any $-2< k
< -1$ or $1 < k<2$, the fiber of $J^{-1}(k)$ is equal to $n$ copies of
$S^2\times S^1$ and for $-1< k < 1$ the fiber $J^{-1}(k)$ is equal to
$2n$ copies of $S^2 \times S^1$. We thank Thomas Baird for this
example.
\end{example}

\begin{remark}
  Much of our interest on this topic came from questions asked by
  physicists and chemists in the context of molecular spectroscopy
  \cite[Section 1]{PeVN2011}, \cite{Fi2009,Sa1995,Ch1999,As2010}.
  Many research teams have been working on this topic, to name a few:
  Mark Child's group in Oxford (UK), Jonathan Tennyson's at the
  University College London (UK), Frank De Lucia's at Ohio State
  University (USA), Boris Zhilinskii's at Dunkerque (France), and Marc
  Joyeux's at Grenoble (France).

  For applications to concrete physical models the theorem is far
  reaching.

 \begin{figure}[h]
   \centering
   \includegraphics[width=10.5cm]{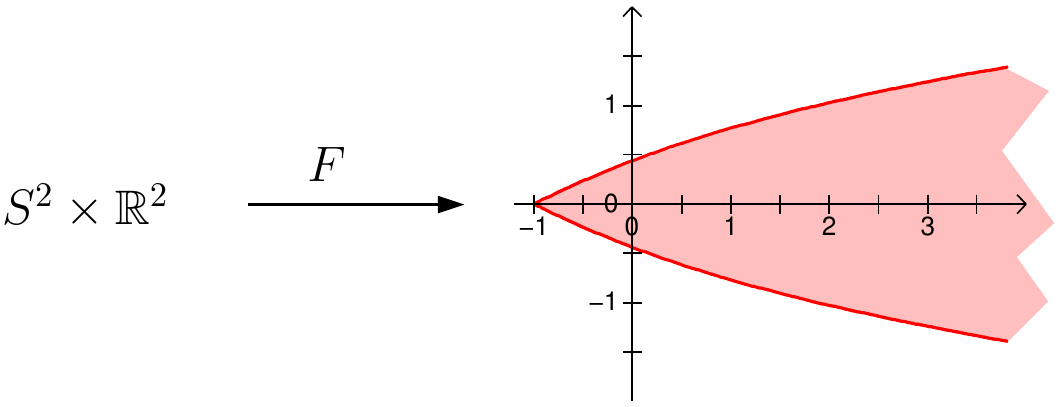}
   \label{fig:example}
   \caption{Coupled-spin oscillator}
 \end{figure}

 The reason is that is essentially impossible to prove connectedness
 of the fibers of concrete physical systems.  For the purpose of
 applications, a computer may approximate the bifurcation set of a
 system and find its image to a high degree of accuracy.  See for
 example the case of the coupled spin-oscillator in Figure
 \ref{fig:example}, which is one of the most fundamental examples in
 classical physics, and which recently has attracted much attention,
 see \cite{Ba2009, Ba2011}.
\end{remark}

\section{The image of an almost\--toric system} \label{sec3}

In this section we study the structure of the image of an almost-toric
system.

\begin{figure} [h]
  \centering
  \includegraphics[width=5.5cm]{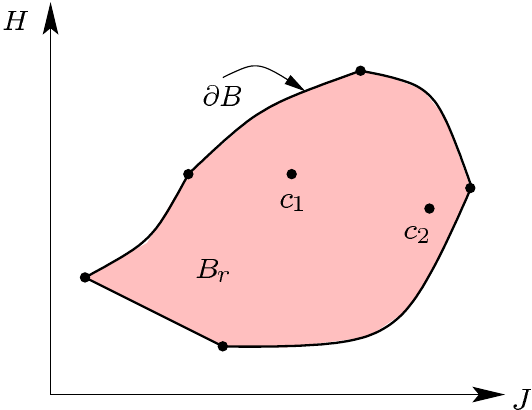}
  \caption{The image $B:=F(M)$ of an almost-toric momentum map
    $F=(J,\,H) \colon M \to \R^2$. The set of regular values of $F$ is
    denoted by $B_r$. The marked dots $c_1,\,c_2$ inside of $B$
    represent singular values, corresponding to the focus-focus
    fibers. The set $B_r$ is equal to $B$ minus $\partial B \cup
    \{c_1,\,c_2\}$.}
  \label{image}
\end{figure}

\subsection{Images bounded by lower/upper semicontinuous graphs}

We star with the following observation.

\begin{lemma} \label{lemma:critical-set} Let $M$ be a connected smooth
  manifold and let $f \colon M \to \mathbb{R}$ be a Morse-Bott
  function with connected fibers. Then the set $C_0$ of index zero
  critical points of $f$ is connected. Moreover, if $\lambda_0:=\inf
  f\geq -\infty$, the following hold:
  \begin{itemize}
  \item[{\rm (1)}] If $\lambda_0>-\infty$ then
    $C_0=f^{-1}(\lambda_0)$.
  \item[{\rm (2)}] If $\lambda_0=-\infty$ then $C_0=\varnothing$.
  \end{itemize}
\end{lemma}

\begin{proof}
  The fiber over a point is locally path connected. If the point is
  critical this follows from the Morse-Bott Lemma and if the point is
  regular, this follows from the submersion theorem.

  Let $m$ be a critical point of index $0$ of $f$, i.e., $m \in C_0$,
  and let $\lambda:=f(m)$, $\Lambda:=f^{-1}(\lambda)$. Let $\gamma
  \colon [0,\,1] \to \Lambda$. Since $\Lambda$ is connected and
  locally path connected, it is path connected.  Let $\gamma \colon
  [0,\,1] \to \Lambda$ be a continuous path starting at $m$. By the
  Morse-Bott Lemma, $\operatorname{im}(\gamma) \subset
  C_0$. Therefore, since $\Lambda$ is path connected,
  $\Lambda\subseteq C_0$.  Each connected component of $C_0$ is
  contained in some fiber of $f$ and hence $\Lambda$ is the connected
  component of $C_0$ that contains $m$. We shall prove that $C_0$ has
  one one connected component.

  Assume that there is a point $m'$ such that $f(m')<\lambda$ and let
  $\delta\colon [0,\,1] \to M$ be a continuous path from $m$ to
  $m'$. Let
$$
t_0:=\inf\{t>0 \, | \, f(\delta(t))<\lambda\}.
$$
Then $f(\delta(t_0))=\lambda$ and, by definition, for every $\alpha>0$
there exists $t_{\alpha} \in [t_0,\, t_0+\alpha]$ such that
$f(\delta(t_{\alpha}))<\lambda$.

Let $m_0:=\delta(t_0)$. Let $U$ be a neighborhood of $\delta(t_0)$ in
which we have the Morse-Bott coordinates given by the Morse-Bott Lemma
centered at $\delta(t_0)$. For $\alpha$ small enough,
$\delta([t_0,\,t_1]) \subset U$ and therefore for $t \in [t_0,\,t_1]$
we have that
$$
f(\delta(t)) =\sum_{i=1}^k y_i^2(\delta(t)) +\lambda \geq \lambda,
$$
which is a contradiction.
\end{proof}

Note that the following result is strictly Morse-theoretic; it does
not involve integrable systems. A version of this result was proven in
\cite[Theorem 3.4]{VN2007} in the case of integrable systems
$F=(J,\,H)$ for which $J \colon M \to \mathbb{R}$ is both a proper map
(hence $F\colon M \to \mathbb{R}^2$ is proper) and a momentum map for
a Hamiltonian $S^1$-action. The version we prove here applies to
smooth maps, which are not necessarily integrable systems.

\begin{theorem} \label{imagemomentum:theorem} Let $M$ be a connected
  smooth four-manifold. Let $F=(J,\,H) \colon M \to \mathbb{R}^2$ be a
  smooth map.  Equip $\overline{\RM}:=\mathbb{R}\cup \{\pm\infty\}$
  with the standard topology. Suppose that the component $J$ is a
  non-constant Morse-Bott function with connected fibers.  Let
  $H^{+},\, H^{-} \colon J(M) \to \overline{\RM}$ be the functions
  defined by $H^{+}(x) := \sup_{J^{-1}(x)} H$ and $H^{-}(x) :=
  \inf_{J^{-1}(x)} H$. The functions $H^{+}$, $-H^{-}$ are lower
  semicontinuous. Moreover, if $F(M)$ is closed in $\mathbb{R}^2$ then
  $H^+$, $-H^{-}$ are upper semicontinuous (and hence continuous), and
  $F(M)$ may be described as
  \begin{eqnarray}
    F(M) = \op{epi}(H^{-}) \cap \op{hyp}(H^+). 
    \label{pp}
  \end{eqnarray}
  In particular, $F(M)$ is contractible.
\end{theorem}

\begin{proof} First we consider the case where $F(M)$ is not
  necessarily closed (Part 1). In Part 2 we prove the stronger result
  when $F(M)$ is closed.

  \paragraph{Part 1.} We do not assume that $F(M)$ is closed and prove
  that $H^+$ is lower semicontinuous; the proof that $H^{-}$ is lower
  semicontinuous is analogous. Since, by assumption, $J$ is
  non-constant, the interior set $\operatorname{int}J(M)$ of $J(M)$ is
  non-empty. The set $\operatorname{int}J(M)$ is an open interval
  $(a,\,b)$ since $M$ is connected and $J$ is continuous.  Lower
  semicontinuity of $H^+$ is proved first in the interior of $J(M)$
  (case A) and then at the possible boundary (case B).
\\
\\
  \emph{Case A.} Let $x_0 \in \operatorname{int}J(M)$, and let
  $y_0:=H^{+}(x_0)$.  Let $\epsilon>0$. By the definition of supremum,
  there exists $\epsilon'>0$ with $\epsilon'<\epsilon$ such that if
  $y_1:=y_0-\epsilon'$ then $F^{-1}(x_0,\,y_1)\neq \varnothing$ (see
  Figure~\ref{fig:lowersemicontinuity}). Here we have assumed that
  $y_0<+\infty$; if $y_0=+\infty$, we just need to replace $y_1$ by an
  arbitrary large constant. Let $m \in F^{-1}(x_0,\,y_1)$. Then
  $J(m)=x_0$. Endow $M$ with a Riemannian metric and, with respect to
  this metric, consider the gradient vector field $\nabla J$ of
  $J$. Let $t_0>0$ such that the flow $\varphi^t(m)$ of $\nabla J$
  starting at $\varphi^0(m)=m$ exists for all $t \in
  (-t_0,\,t_0)$. Now we distinguish two cases.

  \begin{figure}[h]
    \centering
    \includegraphics[width=0.6\textwidth]{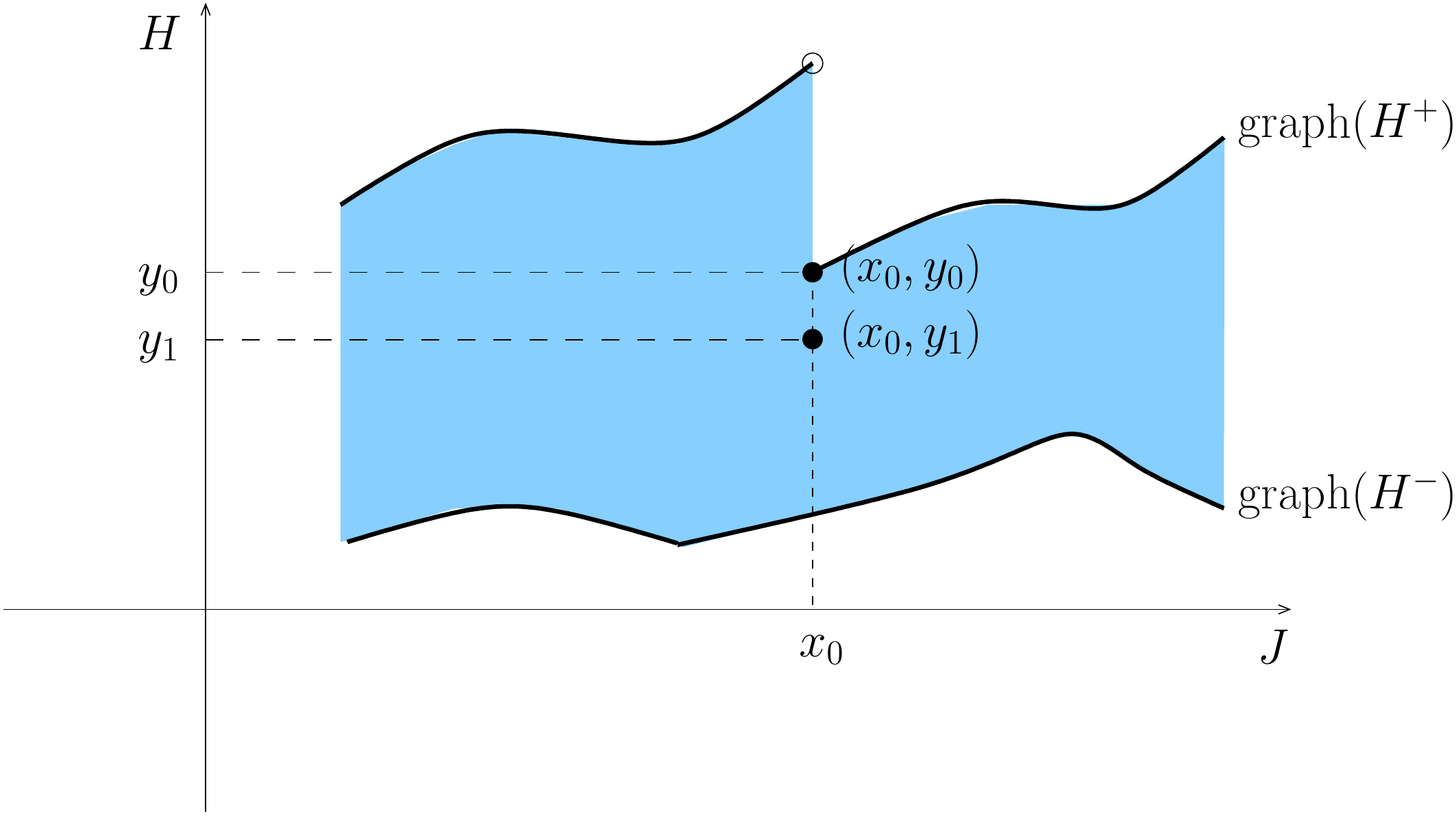}
    \caption{Lower semicontinuity of $H^+$.}
    \label{fig:lowersemicontinuity}
  \end{figure}

  \begin{itemize}
  \item[A.1.]  Assume $\DD J(m)\neq 0$.  Since $\nabla J(m)\neq 0$,
    the set
$$
\Lambda_{t_0}:=\{ J(\varphi(t)) \, | \, t \in (-t_0,\, t_0)\}
$$
is a neighborhood of $x_0$.

Let $B$ the ball of radius $\epsilon$ centered at $(x_0,\, y_1)$. Let
$U:=F^{-1}(B)$, which contains $m$. Let $t'_0\leq t_0$ be small enough
such that $\varphi^t(m)\in U$ for all $|t|<t'_0$. The set
$\Lambda_{t'_0}$ is a neighborhood of $x_0$, so there is $\alpha>0$
such that $(x_0-\alpha, \, x_0+\alpha) \subseteq \Lambda_{t'_0}$.
 
Let $x\in \Lambda_{t'_0}$; there exists $|t|<t'_0$ such that
$J(\varphi^t(m))=x$ by definition of $\Lambda_{t'_0}$. Since
$F(\varphi^t(m)) \in B$ we conclude that $y:=H(\varphi^t(m)) \in
(y_1-\epsilon,\, y_1+\epsilon)$, so $H^+(x)\geq y \geq y_1-\epsilon$
for all $x$ with $|x-x_0|<\alpha$.

Thus we get $H^+(x)\geq y_0-2\epsilon$ for all $x$ with
$|x-x_0|<\alpha$, which proves the lower semicontinuity.

\item[A.2.]  Assume $\DD J(m)=0$. By Lemma~\ref{lemma:critical-set} we
  conclude that $m$ is not of index $0$, for otherwise $J(m)=x_0$
  would be a global minimum in $J(M)$ which contradicts the fact that,
  by assumption, $x_0 \in \operatorname{int}J(M)$.

  Thus the Hessian of $J$ has at least one negative eigenvalue and
  therefore there exists $t_0>0$ such that $\Lambda_{t_0}$ is an open
  neighborhood of $x_0$ and we may then proceed as in Case A.1.
\end{itemize}
Hence $H^+$ is lower semicontinuous.
\\
\\
\emph{Case B.} We prove here lower semicontinuity at a point
$x_0$ in the topological boundary of $J(M)$.  We may assume that
$J(M)=[a,\,b)$ and that $x_0=a$. By Lemma~\ref{lemma:critical-set},
$J^{-1}(a)=C_0$, where $C_0$ denotes the set of critical points of $J$
of index $0$. If $m \in J^{-1}(a)$ and $U$ is a small neighborhood of
$m$, it follows from the Morse-Bott lemma that $J(U)$ is a
neighborhood of $a$ in $J(M)$. Hence we may proceed as in Case A to
conclude that $H^+$ is lower semicontinuous.

\paragraph{Part 2.} Assuming that $F(M)$ is closed we shall prove now
that $H^+$ is upper semicontinuous.

Suppose that $H^+$ is not upper semicontinuous at some point $x_0 \in
J(M)$ (so we must have $H^+(x_0)<\infty$).  Then there exists
$\epsilon_0>0$ and a sequence $\{x_n\} \subset J(M)$ converging to
$x_0$ such that $H^+(x_n)\geq H^+(x_0)+\epsilon_0$.  First assume that
$H^-(x_0)$ is finite. Since $H^-$ is upper semicontinuous, for $x$
close to $x_0$ we have
\[
H^-(x) \leq H^-(x_0)+\frac{\epsilon_0}{2}.
\]
Thus we have
\begin{equation}
  H^-(x_n) \leq H^-(x_0) + \frac{\epsilon_0}{2}\leq
  H^+(x_0)+\frac{\epsilon_0}{2} < H^+(x_0)+\epsilon_0\leq 
  H^+(x_n).
  \label{equ:inequality}
\end{equation}
Let $y_0\in (H^+(x_0)+\frac{\epsilon_0}{2},
H^+(x_0)+\epsilon_0)$. Because of~\eqref{equ:inequality}, and since
$H(J^{-1}(x_n))$ is an interval whose closure is $[H^-(x_n),H^+(x_n)]$
(remember that $J$ has connected fibers and $H$ is continuous), we
must have $y_0\in H(J^{-1}(x_n))$.  Therefore $(x_n,y_0)\in F(M)$.
Since $F(M)$ is closed, the limit of this sequence belongs to $F(M)$;
thus $(x_0,\,y_0) \in F(M)$.  Hence $y_0 \leq H^+(x_0)$, a
contradiction. If $H^-(x_0)=-\infty$, we just replace in the proof
$H^-(x_0) + \frac{\epsilon_0}{2}$ by some constant $A$ such that
$A\leq H^+(x_0)+\frac{\epsilon_0}2$.

Hence $H^{+}$ is upper semicontinuous.  Therefore $H^{+}$ is
continuous. The same argument applies to $H^-$.

In order to prove~\eqref{pp}, notice that for any $x\in J(M)$, we have
the equality
\begin{eqnarray} \label{identity:for} \{x \} \times H(J^{-1}(x))=F(M)
  \cap \{ (x,\,y) \, | \, y \in \mathbb{R}\}.
\end{eqnarray}
Therefore, if $F(M)$ is closed, $\{x \} \times H(J^{-1}(x))$ must be
closed and hence equal to $\{x \} \times [H^-(x),H^+(x)]$.  The
equality \eqref{pp} follows by taking the union of the identity
\eqref{identity:for} over all $x \in J(M)$.

Finally, we show that $F(M)$ is contractible. Since
$\mathbb{R}\cup\{\pm\infty\}$ is homeomorphic to a compact interval,
$F(M)$ is homeomorphic to a closed subset of the strip
$\mathbb{R}\times [-1,1]$ by means of a homeomorphism $g$ that fixes
the first coordinate $x$. Thus, composing~\eqref{pp} by this
homeomorphism we get
\[
g(F(M))=\text{epi}(h^-)\cap \text{hyp}(h^+)
\]
for some continuous functions $h^+$ and $h^-$.  Then the map
\[
g(F(M))\times [0,1] \ni((x,y),t) \mapsto (x,t(y-h^-(x))) \in
\mathbb{R}^2
\]
is a homotopy equivalence with the horizontal axis. Since $g$ is a
homeomorphism, we conclude that $F(M)$ is contractible.
\end{proof}

\subsection{Constructing Morse-Bott functions}

We can give a stronger formulation of Theorem
\ref{imagemomentum:theorem}.  First, recall that if $\Sigma$ is a
smoothly immersed $1$-dimensional manifold in $\mathbb{R}^2$, we say
that $\Sigma$ \emph{has no horizontal tangencies} if there exists a
smooth curve $\gamma \colon I \subset \mathbb{R} \to \mathbb{R}^2$
such that $\gamma(I)=\Sigma$ and $\gamma'_2(t) \neq 0$ for every $t\in
I$. Note that $\Sigma$ has no horizontal tangencies if and only if for
every $c \in \mathbb{R}$ the $1$-manifold $\Sigma$ is transverse to
the horizontal line $y=c$.

We start with the following result, which is of independent interest
and its applicability goes far beyond its use in this paper.

We begin with a description of the structure of the set $\Sigma_F : =
F(\operatorname{Crit}(F))$ of critical values of an integrable systems
$F:M\to\mathbb{R}^2$; as usual $\operatorname{Crit}(F)$ denotes the
set of critical points of $F$.

Let $c_0\in\Sigma_F$ and $B\subset\RM^2$ a small closed ball centered
at $c_0$. For each point $m\in F^{-1}(B)$ we choose a chart about $m$
in which $F$ has normal form (see Theorem
\ref{singularities_theorem}). There are seven types of normal forms,
as depicted in Figure~\ref{singularities_in_image}. Since $F^{-1}(B)$
is compact, we can select a finite number of such chart domains that
still cover $F^{-1}(B)$. For each such chart domain $\Omega$, the set
of critical values of $F|_{\Omega}$ is diffeomorphic to the set of
critical values of one of the models described in
Figure~\ref{singularities_in_image}, which is either empty, an
isolated point, an open curve, or up to four open curves starting from
a common point. Since
$$
\Sigma_F\cap B=F(\textup{Crit}(F)\cap F^{-1}(B)),
$$ 
it follows that $\Sigma\cap B$ is a finite union of such models. This
discussion leads to the following proposition.

\begin{prop}
  \label{prop:strata}
  Let $(M,\, \omega)$ be a connected symplectic four-manifold.  Let
  $F:M\to\RM^2$ be a non-degenerate integrable system.  Suppose that
  $F$ is a proper map.  Then $\Sigma_F:=F(\textup{Crit}(F))$ is the
  union of a finite number of stratified manifolds with 0 and 1
  dimensional strata. More precisely, $\Sigma_F$ is a union of
  isolated points and of smooth images of immersions of closed
  intervals (since $F$ is proper, a 1-dimensional stratum must either
  go to infinity or end at a rank-zero critical value of $F$).
\end{prop}

\begin{definition}
  \label{defi:tangent}
  Let $c\in\Sigma_F$. A vector $v\in\RM^2$ is called \emph{tangent to
    $\Sigma_F$} if there is a smooth immersion
  $\iota:\RM\supset[0,1]\to\Sigma_F$ with $\iota(0)=c$ and
  $\iota'(0)=v$.
\end{definition}
Here $\iota$ is smooth on $[0,1]$, when $[0,1]$ is viewed as a subset
of $\RM$. Notice that a point $c$ can have several linearly
independent tangent vectors.
\begin{definition}
  \label{defi:contact}
  Let $\gamma$ be a smooth curve in $\RM^2$.
  \begin{itemize}
  \item If $\gamma$ intersects $\Sigma_F$ at a point $c$, we say that
    the intersection is \emph{transversal} if no tangent vector of
    $\Sigma_F$ at $c$ is tangent to $\gamma$. Otherwise we say that
    $c$ is a \emph{tangency point}.

  \item Assume that $\gamma$ is tangent to a 1-stratum $\sigma$ of
    $\Sigma_F$ at a point $c\in\sigma$. Near $c$, we may assume that
    $\gamma$ is given by some equation $\phy(x,y)=0$, where
    $\DD\phy(c)\neq 0$.

    We say that $\gamma$ has a \emph{non-degenerate contact} with
    $\sigma$ at $c$ if, whenever $\delta:(-1,1)\to\sigma$ is a smooth
    local parametrization of $\sigma$ near $c$ with $\delta(0)=c$,
    then the map $t\mapsto (\phy\circ\delta)(t)$ has a non-degenerate
    critical point at $t=0$.
  \item Every tangency point that is not a non-degenerate contact is
    called \emph{degenerate} (this includes the case where $\gamma$ is
    tangent to $\Sigma_F$ at a point $c$ which is the end point of a
    1-dimensional stratum $\sigma$).
  \end{itemize}
\end{definition}

With this terminology, we can now see how a Morse function on $\RM^2$
can give rise to a Morse-Bott function on $M$.

\begin{theorem}[Construction of Morse-Bott functions]
  \label{prop:bott}
  Let $(M,\, \omega)$ be a connected symplectic four-manifold.  Let
  $F:=(J,\,H) \colon M \to \mathbb{R}^2$ be an integrable system with
  non-degenerate singularities (of any type, so this statement applies
  to hyperbolic singularities too) such that $F$ is proper.  Let
  $\Sigma_F\subset\RM^2$ be the set of critical values of $F$, i.e.,
  $\Sigma_F:=F(\textup{Crit}(F))$.

  Let $U\subset\RM^2$ be open.  Suppose that $f:U \rightarrow
  \mathbb{R}$ is a Morse function whose critical set is disjoint from
  $\Sigma_F$ and the regular level sets of $f$ intersect $\Sigma_F $
  transversally or with non-degenerate contact.

  Then $f\circ F$ is a Morse-Bott function on $F^{-1}(U)$.
\end{theorem}

Here by regular level set of $f$ we mean a level set corresponding to
a regular value of $f$.

\begin{proof} Let $L=f\circ F$. Writing
  \[
  \DD L =( \partial_1 f )\DD J + (\partial_2 f)\DD H,
  \]
  we see that if $m$ is a critical point of $L$ then either $c=F(m)$
  is a critical point of $f$ (so $\partial_1 f(c)=
  \partial_2f(c)=0$), or $\DD J(m)$ and $\DD H(m)$ are linearly
  dependent (which means $\operatorname{rank}({\rm T}_mF)<2$).  By
  assumption, these two cases are disjoint: if $c=F(m)$ is a critical
  point of $f$, then $c\not\in\Sigma_F$ which means that $m$ is a
  regular point of $F$.

  Thus $\textup{Crit}(L) \subset F^{-1}(\textup{Crit}(f))\sqcup
  \textup{Crit}(F)$ is a disjoint union of two closed sets. Since
  Hausdorff manifolds are normal, these two closed sets have disjoint
  open neighborhoods. Thus $\textup{Crit}(L)$ is a submanifold if and
  only if both sets are submanifolds, which we prove next.

  \paragraph{Study of $F^{-1}(\textup{Crit}(f))$.}

  Let $m_0\in M$ and $c_0=F(m_0)$. We assume that $c_0$ is a critical
  point of $f$, i.e., $\DD f(c_0)=0$. By hypothesis,
  $\operatorname{rank}\DD_{m_0}F=2$. Since the rank is lower
  semicontinuous, there exists a neighborhood $\Omega$ of $m_0$ in
  which $\operatorname{rank}\DD_{m}F=2$ for all $m \in \Omega$.
  Thus, on $\Omega$, $L=f\circ F$ is critical at a point $m$ if and
  only if $F(m)$ is critical for $f$. Since $f$ is a Morse function,
  its critical points are isolated; therefore we can assume that the
  critical set of $L$ in $\Omega$ is precisely
  $F^{-1}(c_0)\cap\Omega$.

  Since $F^{-1}(c_0)$ is a compact regular fiber (because $F$ is
  proper), it is a finite union of Liouville tori (this is the
  statement of the action-angle theorem; the finiteness comes from the
  fact that each connected component is isolated). In particular,
  $F^{-1}( \operatorname{Crit}(f))$ is a submanifold and we can
  analyze the non-degeneracy component-wise.

  Given any $m\in F^{-1}(c_0)$, the submersion theorem ensures that
  $J$ and $H$ can be seen as a set of local coordinates of a
  transversal section to the fiber $F^{-1}(c_0)$. Thus, using the
  Taylor expansion of $f$ of order 2, we get the 2-jet of $L-L(m_0)$:

\[
L(m)-L(m_0) = \frac{1}{2}\operatorname{Hess}f(m_0)(J(m)-J(m_0),
H(m)-H(m_0))^2 + \text{terms of order } 3,
\]
where
\begin{align}
  &\operatorname{Hess}f(m_0)(J(m)-J(m_0),H(m)-H(m_0))^2 \nonumber\\
  &\quad := (\partial^2_1f)(m_0)(J(m)-J(m_0))^2 +
  2(\partial^2_{1,2}f)(m_0)(J(m)-J(m_0)))(H(m)-H(m_0))
  \nonumber\\
  &\qquad\quad + (\partial_2^2f)(m_0)(H(m) - H(m_0))^2.
  \label{equ:hessian-f}
\end{align}
Again, since $(J,H)$ are taken as local coordinates, we see that the
transversal Hessian of $L$ in the $(J,H)$-variables is non-degenerate,
since $\operatorname{Hess}f(m_0)$ is non- degenerate by assumption
($f$ is Morse).

Thus we have shown that $F^{-1}(\textup{Crit}(f))$ is a smooth
submanifold (a finite union of Liouville tori), transversally to which
the Hessian of $L$ is non-degenerate.

\paragraph{Study of $\textup{Crit}(L) \cap \textup{Crit}(F).$}
Let $m_0\in M$ be a critical point for $F$, and let $c_0=F(m_0)$. By
assumption, $c_0$ is a regular value of $f$.  Thus, there exists an
open neighborhood $V$ of $c_0$ in $\RM^2$ that contains only regular
values of $f$. Therefore, the critical set of $L$ in $F^{-1}(V)$ is
included in $\operatorname{Crit}(F)\cap F^{-1}(V)$. In what follows,
we choose $V$ with compact closure in $\RM^2$ and admitting a
neighborhood in the set of regular values of $f$.

\paragraph{\emph{Case 1: rank 1 critical points.}}

There are 2 types of rank 1 critical points of $F$: elliptic and
hyperbolic. By the Normal Form Theorem \ref{singularities_theorem},
there are canonical coordinates $(x_1,x_2,\xi_1,\xi_2)$ at in a chart
about $m_0$ in which $F$ takes the form
\[
F=g(\xi_1,q),
\]
where $q$ is either $x_2^2+\xi_2^2$ (elliptic case) or $x_2\xi_2$
(hyperbolic case), and $g:\RM^2\fleche\RM^2$ is a local diffeomorphism
of a neighborhood of the origin to a neighborhood of $F(m_0)$,
$g(0)=F(m_0)$.

We see from this that
$\Sigma^V:=\operatorname{Crit}\left(F|_{F^{-1}(V)}\right)$ is the
1-dimensional submanifold $\{g(t,0)\mid \abs{t} \text{small}\}$.

Now consider the case when the level sets of $f$ in $V$ (which are
also 1-dimensional submanifolds of $\RM^2$) are transversal to this
submanifold. We see that the range of $\DD_{m_0}F$ is directed
along the first basis vector $e_1$ in $\RM^2$, which is precisely
tangent to $\Sigma^V$.  Hence $\DD f$ cannot vanish on this vector and
hence
$$
0\neq \DD_{F(m_0)}f\circ \DD_{m_0}F= \DD_{m_0} L.
$$ 
This shows that $L$ has no critical points in $F^{-1}(V)$.

Now assume that there is a level set of $f$ in $V$ that is tangent to
$\Sigma^V$ with non-degenerate contact at the point $g(0,0)$. The
tangency gives the equation $\DD_{F(m_0)} f\cdot (\DD_{(0,0)}g
(e_1))=0$.  Since $L=(f\circ g)(\xi_1,q)$, the equation of
$\operatorname{Crit}(L)$ is
\begin{equation}
  \label{equ:critical}
  \deriv{(f\circ g)}{\xi_1}=0 \quad \text{ and } \quad 
  \deriv{(f\circ g)}{q} \DD q =0.
\end{equation}
Since $\DD f\neq 0$ on $V$ and $g$ is a local diffeomorphism, we have
$\DD (f\circ g)\neq 0$ in a neighborhood of the origin. But the
contact equation gives
$$
\deriv{(f\circ g)}{\xi_1}(0,0)=0
$$ 
so, taking $V$ small enough, we may assume that $\deriv{(f\circ
  g)}{q}$ does not vanish. Hence the second condition in
\eqref{equ:critical} is equivalent to $\DD q=0$, which means
$x_2=\xi_2=0$ (and hence $q=0$).

By definition, the contact is non-degenerate if and only if the
function $t\mapsto f(g(t,0))$ has a non-degenerate critical point at
$t=0$. Therefore, by the implicit function theorem, the first equation
\[
\deriv{(f\circ g)}{\xi_1}(\xi_1,0)=0
\]
has a unique solution $\xi_1=0$. Thus, the critical set of $L$ is of
the form $\{(x_1,\xi_1=0, x_2=0,\xi_2=0)\}$, where $x_1$ is arbitrary
in a small neighborhood of the origin; this shows that the critical
set of $L$ is a smooth 1-dimensional submanifold.

It remains to check that the Hessian of $L$ is transversally
non-degenerate. Of course, we take $(\xi_1,x_2,\xi_2)$ as transversal
variables and we write the Taylor expansion of $L$, for any
$m\in\textup{Crit(L)}$:
\begin{equation}
  L= L(m) + \mathcal{O}(x_1) + \deriv{(f\circ g)}{q} q +
  \frac{1}{2}\frac{\partial^2 (f\circ g)}{\partial \xi_1^2} 
  \xi_1^2 + \mathcal{O}((\xi_1,x_2,\xi_2)^3).
  \label{equ:hessian-rank-one}
\end{equation}
We know that $\deriv{(f\circ g)}{q}\neq 0$ and, by the non-degeneracy
of the contact,
$$
\frac{\partial^2 (f\circ g)}{\partial \xi_1^2}\neq 0.
$$ Recalling that $q=x_2\xi_2$ or
$q=x_2^2+\xi_2^2$, we see that the $(\xi_1,x_2,\xi_2)$-Hessian of $L$
is indeed non-degenerate.

\paragraph{\emph{Case 2: rank 0 critical points.}}
There are 4 types of rank 0 critical point of $F$: elliptic-elliptic,
focus-focus, hyperbolic-hyperbolic, and elliptic-hyperbolic, giving
rise to four subcases.  From the normal form of these singularities
(see Theorem \ref{singularities_theorem}, we see that all of them are
isolated from each other. Thus, since $F$ is proper, the set of rank 0
critical points of $F$ is finite in $F^{-1}(V)$.

Again, let $m_0$ be a rank 0 critical point of $F$ and $c_0:=F(m_0)$.

\begin{enumerate}[(a)]

\item{ \emph{Elliptic-elliptic subcase.}}

  In the elliptic-elliptic case, the normal form is
  \[
  F=g(q_1,q_2),
  \]
  where $q_i=(x_i^2+\xi_i^2)/2$.  The critical set of $F$ is the union
  of the planes $\{z_1=0\}$ and $\{z_2=0\}$ (we use the notation
  $z_j=(x_j,\xi_j)$). The corresponding critical values in $V$ is the
  set
$$
\Sigma^V:=\{g(x=0,y\geq 0)\}\cup\{g(x\geq 0,y=0)\}
$$ 
(the $g$-image of the closed positive quadrant).
 
The transversality assumption on $f$ amounts here to saying that the
level sets of $f$ in a neighborhood of $c_0$ intersect $\Sigma^V$
transversally; in other words, the level sets of $h:=f\circ g$
intersect the boundary of the positive quadrant transversally. Up to
further shrinking of $V$, this amounts to requiring $\DD_z h(e_1)\neq
0$, $\DD_z h(e_2)\neq 0$ for all $z\in g^{-1}(V)$, where $(e_1,e_2)$
is the canonical $\RM^2$-basis.

Any critical point $m$ of $F$ different from $m_0$ is a rank 1
elliptic critical point.  Since the level sets of $f$ don't have any
tangency with $\Sigma^V$, we know from the rank 1 case above that $m$
cannot be a critical point of $L$. Hence $m_0$ is an isolated critical
point for $L$.

The Hessian of $L$ at $m_0$ is calculated via the normal form: it has
the form $aq_1+bq_2$, with $a=\DD_0h(e_1)$ and $b=\DD_0 h(e_2)$.  The
Hessian determinant is $a^2b^2$. The transversality assumption implies
that both $a$ and $b$ are non-zero which means that the Hessian is
non-degenerate.

\item{\emph{Focus-focus subcase.}}

  The focus-focus critical point is isolated, so we just need to prove
  that the Hessian of $L$ is non-degenerate. But the 2-jet of $L$ is
  \begin{align*}
    L(m)-L(m_0) &= (\partial_1 f)(F(m_0)(J(m) - J (m_0))\\
    & \qquad + (\partial_2 f)(F(m_0)(H(m) - H(m_0)) + \text{terms of
      order } 3.
  \end{align*}
  Thus, in normal form coordinates (see Theorem
  \ref{singularities_theorem}), as in the previous case, it has the
  form $a q_1 + b q_2$, where this time $q_1$ and $q_2$ are the
  focus-focus quadratic forms given in
  Theorem~\ref{singularities_theorem}(iii). The Hessian determinant is
  now $(a^2+b^2)^2$, which does not vanish.

\item{\emph{Hyperbolic\--hyperbolic subcase.}}

  Here the local model for the foliation is $q_1=x_1\xi_1$,
  $q_2=x_2\xi_2$. However, the formulation $F=g(q_1,q_2)$ may not
  hold; this is a well-known problem for hyperbolic fibers.
  Nevertheless, on each of the 4 connected components of of
  $\RM^4\setminus(\{x_1=0\}\cup \{x_2=0\})$, we have a diffeomorphism
  $g_i$, $i={1,2,3,4}$ such that $F=g_i(q_1,q_2)$. These four
  diffeomorphisms agree up to a flat map at the origin (which means
  that their Taylor series at $(0,0)$ are all the same).

  Thus, the critical set of $F$ in these local coordinates is the
  union of the sets $\{q_1=0\}$ and $\{q_2=0\}$: this is the union of
  the four coordinate hyperplanes in $\RM^4$.  The corresponding set
  of critical values in $V$ is the image of the coordinate axes:
  \[
  \Sigma^V:=\bigcup_{i=1,2,3,4}\{g_i(0,y)\}\cup\{g_i(x,0)\},
  \]
  where $x$ and $y$ both vary in a small neighborhood of the origin in
  $\RM$.

  For each $i$ we let $h_i:=f\circ g_i$. As before, the transversality
  assumption says that the values $\DD_0 h_i(e_1)$ and $\DD_0
  h_i(e_2)$ (which don't depend on $i=1,2,3,4$ at the origin of
  $\RM^2$) don't vanish in $V$. Thus, as in the elliptic-elliptic
  case, the level sets of $f$ don't have any tangency with
  $\Sigma^V$. Hence no rank 1 critical point of $F$ can be a critical
  point of $L$, which shows that $m_0$ is thus an isolated critical
  point of $L$.

  The Hessian determinant of $aq_1+bq_2$ is again $a^2b^2$ with $a
  \neq 0$ and $b \neq 0$; thus the Hessian of $L$ at $m_0$ is
  non-degenerate.

\item{\emph{Hyperbolic\--elliptic subcase.}}

  We still argue as above. However, the Hessian determinant in this
  case is $-a^2b^2 \neq 0$.
\end{enumerate}

Summarizing, we have proved that rank 0 critical points of $F$
correspond to isolated critical points of $L$, all of them
non-degenerate.

Putting together the discussion in the rank 1 and 0 cases, we have
shown that the critical set of $L$ consist of isolated non-degenerate
critical points and isolated 1-dimensional submanifolds on which the
Hessian of $L$ is transversally non-degenerate. This means that $L$ is
a Morse-Bott function.
\end{proof}

\subsection{Contact points and Morse-Bott indices}

Since we have calculated all the possible Hessians, it is easy to
compute the various indices that can occur. We shall need a particular
case, for which we introduce another condition on $f$.

\begin{definition}\label{defi:outward}
  Let $(M,\, \omega)$ be a connected symplectic four-manifold.  Let
  $F:M\to\RM^2$ be an almost-toric system with critical value set
  $\Sigma_F$.  A smooth curve $\gamma$ in $\RM^2$ is said to have an
  \emph{outward contact} with $F(M)$ at a point $c\in F(M)$ when there
  is a small neighborhood of $c$ in which the point $\{c\}$ is the
  only intersection of $\gamma$ with $F(M)$.
\end{definition}
\begin{figure}[h]
  \centering
  \includegraphics[width=0.4\textwidth]{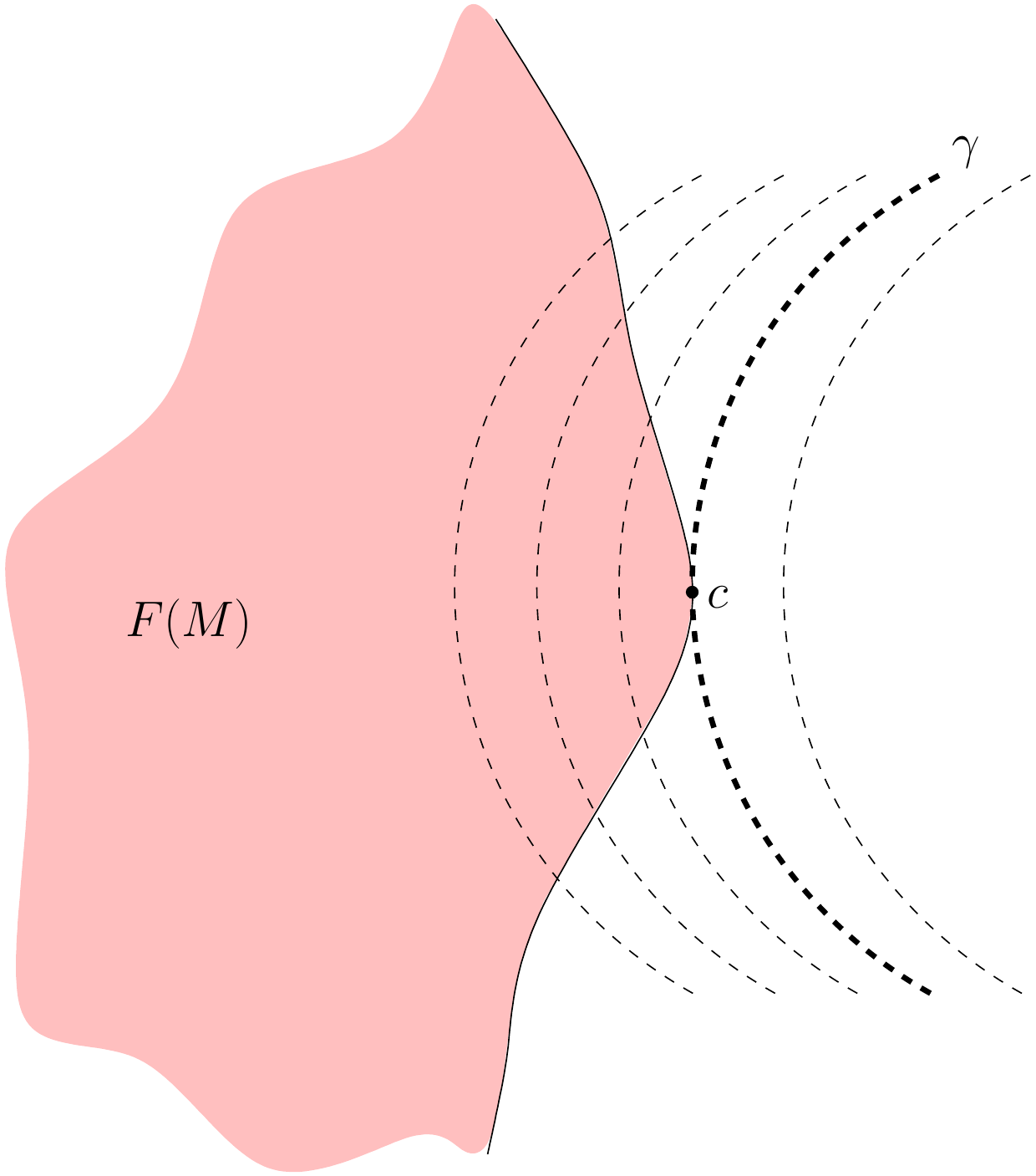}
  \caption{An outward contact point.}
  \label{fig:contact}
\end{figure}
In the proof below we give a characterization in local coordinates.

\begin{prop} \label{prop:bott-index} Let $(M,\, \omega)$ be a
  connected symplectic four-manifold.  Let $F:M\to\RM^2$ be an
  almost-toric system with critical value set $\Sigma_F$. Let $f$ be a
  Morse function defined on an open neighborhood of $F(M)\subset\RM^2$
  such that
  \begin{enumerate}[\rm (i)]
  \item \label{bott-index-hyp1} The critical set of $f$ is disjoint
    from $\Sigma_F$;
  \item \label{bott-index-hyp2} $f$ has no saddle points in $F(M)$;
  \item \label{bott-index-hyp3} the regular level sets of $f$
    intersect $\Sigma_F$ transversally or have a non-degenerate
    outward contact with $F(M)$. (See definitions \ref{defi:contact}
    and \ref{defi:outward}.)
  \end{enumerate}
  Then $f\circ F:M\to \RM$ is a Morse-Bott function with all indices
  and co-indices equal to 0, 2, or 3.
\end{prop}
\begin{proof}
  Because of Theorem \ref{prop:bott}, we just need to prove the
  statement about the indices of $f$.  At points of
  $F^{-1}(\textup{Crit}(f))$, we saw in \eqref{equ:hessian-f} that the
  transversal Hessian of $f\circ F$ is just the Hessian of $f$. By
  assumption, $f$ has no saddle point, so its (co)index is either $0$
  or $2$.  We analyze the various possibilities at points of
  $\textup{Crit}(F)$. There are two possible rank 0 cases for an
  almost-toric system: elliptic-elliptic and focus-focus.  At such
  points, the Hessian determinant is positive (see Theorem
  \ref{singularities_theorem}), so the index and co-index are even.

  In the rank 1 case, for an almost-toric system, only transversally
  elliptic singularities are possible.  We are interested in the case
  of a tangency (otherwise $f\circ F$ has no critical point). The
  Hessian is computed in \eqref{equ:hessian-rank-one} and we use below
  the same notations.  The level set of $f$ through the tangency point
  is given by $f(x,y)=f(g(0,0))$. We switch to the coordinates
  $(\xi_1,q)=g^{-1}(x,y)$, where the local image of $F$ is the
  half-space $\{q\geq 0\}$. Let
$$
h=f\circ g-f(g(0,0)).
$$ 
The level set of $f$ is $h(\xi_1,q)=0$, and $h$ satisfies $\DD_0 h\neq
0$, $\deriv{h}{\xi_1}(0,0)=0$, $\frac{\partial^2
  h}{\partial\xi_1^2}(0,0)\neq 0$ (this is the non-degeneracy
condition in Definition \ref{defi:contact}).  By the implicit function
theorem, the level set $\{h=0\}$ near the origin is the graph
$\{(\xi_1,q)\mid q=\phy(\xi_1)\}$, where
\[
\phy'(0)=0, \qquad \phy''(0) = \frac{-\frac{\partial^2
    h}{\partial\xi_1^2}(0,0)} {\deriv{h}{q_2}(0,0)}\;.
\]
This level set has an outward contact if and only if $\phy''(0)<0$ or,
equivalently, $\frac{\partial^2 h}{\partial\xi_1^2}(0,0)$ and
$\deriv{h}{q_2}(0,0)$ have the same sign. From
\eqref{equ:hessian-rank-one} we see that the index and coindex can
only be 0 or 3.
\end{proof}

\subsection{Proof of Theorem \ref{theo:main-connectivity} and Theorem
  \ref{theo:main-image}}

We conclude by proving the two theorems in the introduction. Both will
rely on the following result. In the statement below, we use the
stratified structure of the bifurcation set of a non-degenerate
integrable system, as given by Proposition~\ref{prop:strata}.

\begin{prop} \label{prop:J-connected} Let $(M,\, \omega)$ be a
  connected symplectic four-manifold.  Let $F\colon M \to
  \mathbb{R}^2$ be an almost\--toric system such that $F$ is proper.
  Denote by $\Sigma_F$ the bifurcation set of $F$. Assume that there
  exists a diffeomorphism $g \colon F(M) \to \mathbb{R}^2$ onto its
  image such that:
  \begin{enumerate}[{\rm (i)}]
  \item \label{hyp:cone} $g(F(M))$ is included in a proper convex cone
    $C_{\alpha,\, \beta}$ (see Figure~\ref{fig:cone}).
  \item \label{hyp:tangencies} $g(\Sigma_F)$ does not have vertical
    tangencies (see Figure~\ref{fig:diffeo}).
  \end{enumerate}
  Write $g\circ F = (J,H)$. Then $J$ is a Morse-Bott function with
  connected level sets.
\end{prop}
\begin{proof}
  Let $\tilde{F}:=g\circ F$. The set of critical values of $\tilde{F}$
  is $\tilde{\Sigma}=g(\Sigma_F)$. We wish to apply
  Proposition~\ref{prop:bott-index} to this new map $\tilde{F}$.  Let
  $f:\RM^2\to\RM$ be the projection on the first coordinate:
  $f(x,y)=x$, so that $f\circ \tilde{F}=J$. Since $f$ has no critical
  points, it satisfies the hypotheses~\eqref{bott-index-hyp1}
  and~\eqref{bott-index-hyp2} of
  Proposition~\ref{prop:bott-index}. The regular levels sets of $f$
  are the vertical lines, and the fact that $\tilde{\Sigma}$ has no
  vertical tangencies means that the regular level sets of $f$
  intersect $\tilde{\Sigma}$ transversally. Thus the last
  hypothesis~\eqref{bott-index-hyp3} of
  Proposition~\ref{prop:bott-index} is fulfilled and we conclude that
  $J$ is a Morse-Bott function whose indices and co-indices are always
  different from 1.

  Now, since $\tilde{F}$ is proper, the fact that $\tilde{F}(M)$ is
  included in a cone $C_{\alpha,\beta}$ easily implies that $f\circ
  \tilde{F}$ is proper.  Thus, using Proposition~\ref{prop:levelset},
  we conclude that $J$ has connected level sets.
\end{proof}

\begin{proof}[Proof of theorem~\ref{theo:main-connectivity}]
  By Proposition \ref{prop:J-connected}, we get that $J$ has connected
  level sets.  It is enough to apply Theorem \ref{theo:fibers} to
  conclude that $\tilde{F}$ (and thus $F$) has connected fibers.
\end{proof}

\begin{proof}[Proof of theorem~\ref{theo:main-image}]
  Using again Proposition \ref{prop:J-connected}, we conclude that $J$
  is a Morse-Bott function with connected level sets.  By the
  definition of an integrable system, $J$ cannot be constant (its
  differential would vanish everywhere). Thus we can apply Theorem
  \ref{imagemomentum:theorem}, which yields the desired conclusion.
\end{proof}

It turns out that even in the compact case,
Theorem~\ref{theo:main-connectivity} has quite a striking corollary,
which we stated as Theorem~\ref{theo:striking} in the introduction.

\begin{proof}[Proof of Theorem~\ref{theo:striking}]
  The last two cases (a disk with two conic points and a polygon) can
  be transformed by a diffeomorphism as in
  Theorem~\ref{theo:main-connectivity0} to remove vertical tangencies,
  and hence the theorem implies that the fibers of $F$ are connected.

  For the first two cases, we follow the line of the proof of
  theorem~\ref{theo:main-connectivity0}. The use of
  Proposition~\ref{prop:bott-index} is still valid for the same
  function $f(x,y)=x$ even if now the level sets of $f$ can be tangent
  to $\Sigma_F$. Indeed one can check that here only non-degenerate
  outward contacts occur. Then one can bypass
  Proposition~\ref{prop:J-connected} and directly apply
  Proposition~\ref{prop:levelset}. Therefore the conclusion of the
  theorem still holds.
\end{proof}

\section{The spherical pendulum} \label{sp:section}

The goal of this section is to prove that the spherical pendulum is a
non-degenerate integrable system to which our theorems apply. The
configuration space is $S^2$.

We identify the phase space ${\rm T}^\ast S^2$ with the tangent bundle
${\rm T}S^2$ using the standard Riemannian metric on $S^2$ naturally
induced by the inner product on $\mathbb{R}^3$. Denote the points in
$\mathbb{R}^3$ by $\mathbf{q}$. The conjugate momenta are denoted by
$\mathbf{p}$ and hence the canonical one- and two-forms are $p_i\DD
q^i$ and $\DD q^i\wedge \DD p _i$, respectively.  The manifold ${\rm
  T}S ^2$ has its own natural exact symplectic structure.  It is easy
to see that $\iota:{\rm T} S ^2 \hookrightarrow \mathbb{R}^3$ is a
symplectic embedding since $\iota^\ast\left(p_i\DD q^i\right)$
coincides with the canonical one-form on ${\rm T}S^2$. The action
given by rotations about the $\mathbf{k}$-axis is given by
\begin{equation}
  \label{action_formula}
  \mathbf{q}\in \mathbb{R}^3 \mapsto
  \begin{bmatrix}
    \cos \psi &-\sin \psi&0\\
    \sin \psi & \cos \psi& 0\\
    0&0& 1
  \end{bmatrix} \mathbf{q} \in \mathbb{R}^3, \qquad \psi \in
  \mathbb{R}. \nonumber
\end{equation}
The equations of motion determined by the infinitesimal generator of
the Lie algebra element $1 \in \mathbb{R}$ for the lifted action to
${\rm T} \mathbb{R}^3 = \mathbb{R}^3\times \mathbb{R}^3$ are
\begin{equation}
  \label{equ_momentum}
  \dot{\mathbf{q}} = - \mathbf{q} \times \mathbf{k}, \qquad 
  \dot{\mathbf{p}} = - \mathbf{p}\times \mathbf{k}.
\end{equation}
The $S^1$-invariant momentum map of this action is given by (see,
e.g., \cite[Theorem 12.1.4]{MaRa2003}) $J_{T \mathbb{R}^3}(\mathbf{q},
\mathbf{p}) = (\mathbf{q}\times \mathbf{p}) \cdot \mathbf{k} = q^1p_2
- q^2p_1$. Note that ${\rm T}S^2$ is invariant under the flow of
\eqref{equ_momentum} and thus $J:=J_{{\rm T} \mathbb{R}^3}|_{{\rm
    T}S^2}:{\rm T}S^2 \rightarrow \mathbb{R}$ given for all
$(\mathbf{q}, \mathbf{p}) \in {\rm T}S^2$ is the momentum map of the
$S^1$-action on ${\rm T}S^2$. In particular, the equations of motion
of the Hamiltonian vector field $\mathcal{H}_J $ on ${\rm T}S^2$ are
\eqref{equ_momentum}.  We also have $J({\rm T}S^2) = \mathbb{R}$;
indeed, if $\mathbf{q}=(1,0,0)$ and $\mathbf{p}=(0,p_2, p_3)$, then
$$
J(\mathbf{q}, \mathbf{p}) = q^1p_2 - q^2p_1 = p_2
$$ 
is an arbitrary element of $\mathbb{R}$. The momentum map $J$ is
\textit{not} proper: the sequence $(\mathbf{q}_n, \mathbf{p}_n):=
((0,0,1), (n,n,0))\in {\rm T} S^2$ does not contain any convergent
subsequence and the sequence of images $J(\mathbf{q}_n, \mathbf{p}_n)
=0$ is constant, hence convergent.

Let us describe the equations of motion. The Hamiltonian of the
spherical pendulum is
\begin{equation}
  H(\mathbf{q}, \mathbf{p})= \frac{1}{2}\|\mathbf{p}\|^2 +
  \mathbf{q}\cdot \mathbf{k}, (\mathbf{q},\mathbf{p}) \in {\rm T}S^2,
  \label{equ:H}
\end{equation}
where $\mathbf{i}, \mathbf{j}, \mathbf{k}$ is the standard orthonormal
basis of $\mathbb{R}^3$ with $\mathbf{k}$ aligned with and pointing
opposite the direction of gravity and we set all parametric constants
equal to one.  The equations of motion for $(\mathbf{q}, \mathbf{p})
\in {\rm T}S^2$ are given by
\begin{equation}
  \left\{
    \label{equ_shperical_pendulum}
    \begin{aligned}
      \dot{\mathbf{q}}&= \mathbf{p} - \frac{\mathbf{q}\cdot
        \mathbf{p}}{\|\mathbf{q}\|^2}\mathbf{q}
      = \mathbf{p}\\
      \dot{\mathbf{p}}&=-\mathbf{k} -
      \frac{\|\mathbf{p}\|^2}{\|\mathbf{q}\|^2} \mathbf{q} +
      \frac{\mathbf{q}\cdot\mathbf{k}}{\|\mathbf{q}\|^2}\mathbf{q} +
      \frac{\mathbf{q}\cdot \mathbf{p}}{\| \mathbf{q}\|^2} \mathbf{p}
      = - \mathbf{k} + \left(\mathbf{q} \cdot \mathbf{k}-
        \|\mathbf{p}\|^2\right) \mathbf{q}.
    \end{aligned}
  \right.
\end{equation}

\begin{figure}[h]
  \centering
  \includegraphics[width=0.35\textwidth]{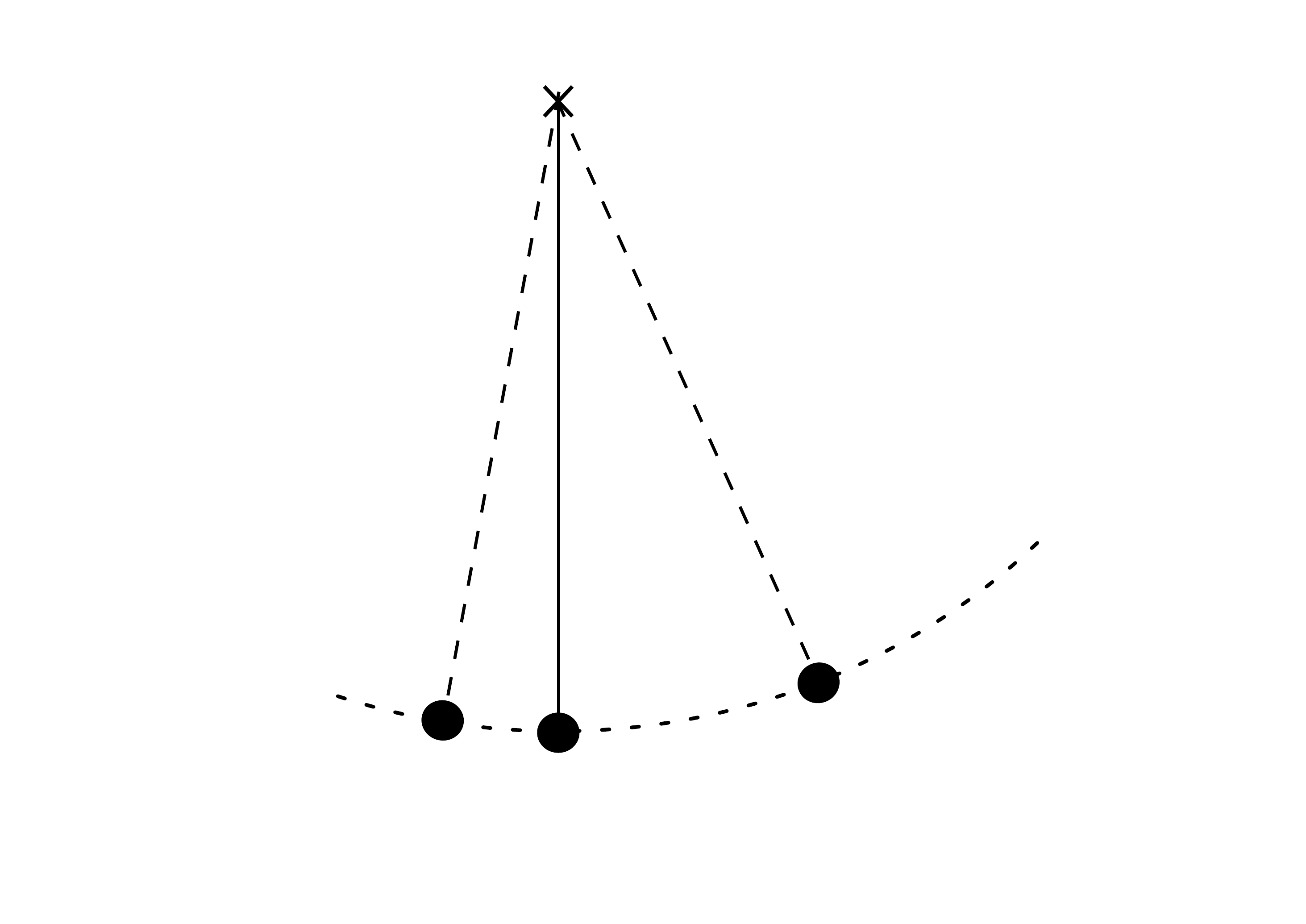}
  \caption{The spherical pendulum.}
\end{figure}

\begin{theorem}
  \label{sphericalpendulum:theorem}
  Let $F:=(J,\,H) \colon M \to \R^2$ be the singular fibration
  associated with the spherical pendulum.
  \begin{itemize}
  \item[{\rm (1)}] $F$ is an integrable system.
  \item[{\rm (2)}] The singularities of $F$ are non\--degenerate.
  \item[{\rm (3)}] The singularities of $F$ are of focus-focus,
    elliptic-elliptic, or transversally elliptic-type so, in
    particular, $F$ has no hyperbolic singularities.  There is
    precisely one elliptic-elliptic singularity at
    $((0,0,-1),(0,0,0))$, one focus-focus singularity at
    $((0,0,1),(0,0,0))$, and uncountably many singularities of
    transversally-elliptic type.
  \item[{\rm (4)}] $H$ is proper and hence $F$ is also proper, even
    though $J$ is \emph{not} proper.
  \item[{\rm (5)}] The critical set of $F$ and the bifurcation set of
    $F$ are equal and given in Figure
    \ref{critical_set_spherical_pendulum.figure}: it consists of the
    boundary of the planar region therein depicted and the interior
    point which corresponds to the image of the only focus-focus point
    of the system; see point (3) above.
  \item[{\rm (6)}] The fibers of $F$ are connected.
  \item[{\rm (7)}] The range of $F$ is equal to planar region in
    Figure \ref{critical_set_spherical_pendulum.figure}. The image
    under $F$ of the focus-singularity is the point $(0,\,1)$.  The
    image under $F$ of the elliptic-elliptic singularity is the point
    $(0,\,-1)$.
  \end{itemize}
\end{theorem}

\begin{proof}
  We prove each item separately.

  \paragraph{{\em {\rm (1)} Integrability.}} Since the Hamiltonian $H$
  is $S^1$-invariant, the associated momentum map $J $ is conserved,
  i.e., $\{H, J\} =0$. To prove integrability, we need to show that
  $\DD H $ and $\DD J $ are linearly independent almost everywhere on
  ${\rm T}S^2$. Note that if $(\mathbf{q},\mathbf{p})\in {\rm T}S^2$,
  the one-forms $\DD H(\mathbf{q},\mathbf{p})$ and $\DD
  J(\mathbf{q},\mathbf{p})$ are linearly dependent precisely when the
  vector fields \eqref{equ_shperical_pendulum} and
  \eqref{equ_momentum} are linearly dependent. Thus we need to
  determine all $\mathbf{q}, \mathbf{p}\in \mathbb{R}^3$ such that
  there exist $a,b \in \mathbb{R}$, not both zero, satisfying
  $a\mathbf{p} - b \mathbf{q}\times\mathbf{k} = 0$, $-a \mathbf{k} +
  a\left(\mathbf{q}\cdot\mathbf{k} - \|\mathbf{p}\|^2\right)\mathbf{q}
  - b \mathbf{p}\times \mathbf{k} =0$, $\|\mathbf{q}\|^2 = 1$,
  $\mathbf{q}\cdot \mathbf{p}=0$.  A computation gives that the set of
  points $(\mathbf{q}, \mathbf{p}) \in {\rm T}S^2$ for which $\DD H $
  and $\DD J $ are linearly dependent is the measure zero set:
  \begin{equation}
    \label{critical_points}
    \left\{((0,0,1), (0,0))\right\}\cup
    \left\{\left.\left(\mathbf{q}, \frac{1}{\lambda} \mathbf{q}
          \times \mathbf{k}\right) \right|\; q^3=- \lambda^2,\; 
      (q^1)^2 + (q^2)^2 = 1 - \lambda^4, \; 0<|\lambda|\leq 1\right\}.
  \end{equation}

  \paragraph{{\em {\rm (2)} Non-degeneracy.}} All critical points of
  $F:=(J,H): {\rm T}S^2 \rightarrow \mathbb{R}^2$ are
  non-degenerate. We prove it for rank 1 critical points.  The
  non-degeneracy at the rank 0 critical points $((0,0,1),(0,0,0))$ and
  $((0,0,-1),(0,0,0))$ are similar exercices (the computation for the
  point $((0,0,1), (0,0,0))$ may also be found in \cite{san-mono}). So
  we consider the singularities $(\mathbf{q}_0,\mathbf{p}_0)$ given by
  $-q^3_0=\lambda^2$, $\mathbf{p}_0= \frac{1}{\lambda}\mathbf{q}_0
  \wedge {\bf k}$, where $\lambda \in (0,\,1)$ (and hence $q^3_0 \in
  (-1,\,0)$).  We have $\DD J(\mathbf{q}_0,\,\mathbf{p}_0) \neq 0$ and
  $\DD H(\mathbf{q}_0,\,\mathbf{p}_0) \neq 0$.  As explained at the
  beginning of the section, the component $J$ is a momentum map for a
  Hamiltonian $S^1$-action. The flow of $J$ rotates about the
  $\mathbf{k}$ axis, so any vertical plane is transversal to the flow
  in the $\mathbf{q}$-coordinates.  The surface $\Sigma$ obtained by
  intersecting the level set of $J$ with the plane $q_2=0$ is chosen
  as the symplectic transversal surface to the flow of
  $\mathcal{H}_H$. The restriction of $J(\mathbf{q}, \mathbf{p})
  =q^1p_2 - q^2p_1$ to $\Sigma$ is $J|_{\Sigma}(\mathbf{q},
  \mathbf{p})=q^1p_2$, so that $q^1p_2$ is constant on $\Sigma$.

  We work near $(\mathbf{q}_0,\,\mathbf{p}_0)$, so $q^1 \neq 0$, and
  hence on $\Sigma$ we have
  \begin{eqnarray} \label{para} p_2=J/q^1,\,\,\,\,\,\,
    q^3=-\sqrt{1-(q^1)^2},\,\,\,\,\,\, q^1p_1+q^3p_3=0\,\,\,
    \Rightarrow\,\,\, p_3=\frac{q^1p_1}{\sqrt{1-(q^1)^2}}.
  \end{eqnarray}
  So $\Sigma$ may be smoothly parametrized by the coordinates
  $(q^1,p_1)$ using \eqref{para}.  Let us compute the Hessian of $H$
  at $(\mathbf{q}_0,\, \mathbf{p}_0)$ in these coordinates.  The
  Taylor expansion of $H$ at $(\mathbf{q}_0,\, \mathbf{p}_0)$ is
$$
H=\frac{1}{2}(p_1^2+p_2^2+p_3^2)+q_3 =\frac{1}{2}\left(p_1^2
  +\frac{J}{(q^1)^2}+\frac{(q^1)^2p_1^2}{1-(q^1)^2}\right)
-\sqrt{1-(q^1)^2}.
$$
By \eqref{critical_points} we get
$(p_1^0,\,p_2^0,\,p_3^0)=(q_2^0,\,-q_1^0,\,0)/\lambda$ and hence the
critical point is
\[
(\mathbf{q}_0,\mathbf{p}_0)=
\left(q^1_0,\,0,\,q^3_0=\sqrt{1-(q^1_0)^2},\, 0,\, p_2^0=J/q^1_0,\,
  p_3^0=0 \right).
\]
To simplify notation, let us denote $s=(q^1)^2 \in [0,1]$; thus
$(s,p_1)$ are the smooth local coordinates near $(\mathbf{q}_0,
\mathbf{p}_0) \simeq (s_0, \mathbf{p}_1^0)$ and the expression on the
Hamiltonian is
\begin{equation}
  \label{abc}
  H=\frac{1}{2}\left(p_1^2+\frac{J^2}{s}+ \frac{sp_1^2}{1-s}
  \right)-\sqrt{1-s}
  = p_1^2\left(\frac{1}{2}+\frac{s}{2(1-s)}\right)+
  \frac{J^2}{2s}-\sqrt{1-s}
\end{equation}
The first term in \eqref{abc} is equal to
$$
p_1^2\left(\frac{1}{2}+\frac{s_0}{2(1-s_0)}\right)
+\mathcal{O}((s-s_0, p_1)^3),$$ while the Taylor expansion of the
second term $\Big( \frac{J^2}{2s}-\sqrt{1-s}\Big)$ with respect to $s$
is
$$
\left(\frac{J^2}{2s_0}-\sqrt{1-s_0}\right)+(s-s_0)
\left(-\frac{J^2}{2s_0^2}+\frac{1}{2\sqrt{1-s_0}}\right) +(s-s_0)^2
\left(\frac{J^2}{s_0^3} + \frac{1}{4(1-s_0)^{3/2}}
\right)+\mathcal{O}((s-s_0)^3).
$$
The coefficient
$\left(-\frac{J^2}{2s_0^2}+\frac{1}{2\sqrt{1-s_0}}\right)$ vanishes
since $\DD H|_{\Sigma}(\mathbf{q}_0,\, \mathbf{p}_0)=0$, so the
Hessian $\operatorname{Hess}(H|_{\Sigma})(s,p_1)$ of $H_{\Sigma}$ is
of the form $Ap_1^2+B(s-s_0)^2$, where $A>0$ and $B>0$, which is
non-degenerate, as we wanted to show.

\paragraph{{\em {\rm (3)} The nature of the singularities of $F$.}}
The statement in the theorem follows from the computations in (2).
The equilibrium $((0,0,1),(0,0,0))$ is of focus-focus type and the
equilibrium $((0,0,-1),(0,0,0))$ is of elliptic-elliptic type.  The
other critical points are of transversally-elliptic type.

\paragraph{{\em {\rm (4)} $H$ and $F$ are proper maps.}}
The properness of $H$ follows directly from the defining
formula~\eqref{equ:H}; indeed, on ${\rm T}S^2$, the map
$(\mathbf{q},\mathbf{p})\mapsto \mathbf{q}\cdot \mathbf{k}$ takes
values in a compact set, hence if $K\subset\RM$ is compact, there is a
compact set $K'\subset\RM$ such that $H^{-1}(K)$ is closed in
$\{(\mathbf{q},\mathbf{p}) \mid \norm{\mathbf{p}}^2\in K'\}$ and hence
is compact. Thus $H$ is proper.

Then for any compact set $C\subset\RM^2$, $F^{-1}(C)\subset
H^{-1}(\operatorname{pr}_2(C))$ is compact, thus $F$ is proper as
well. ($\operatorname{pr}_2: \mathbb{R}^2\rightarrow \mathbb{R}$ is
the projection on the second factor).

On the other hand it is clear that the level sets of $J$ are
unbounded, so $J$ cannot be proper.

\paragraph{{\em {\rm (5)} Critical and bifurcation sets.}} The
critical set is the image of \eqref{critical_points} by the map
$(J,\,H): {\rm T}S^2 \rightarrow \mathbb{R}^2$. We have $J((0,0,\pm
1),(0,0,0))=0$, $H((0,0,\pm 1), (0,0,0)) = \pm 1$.  In addition,
\begin{equation} \nonumber
  \label{parametric_rep_sing_curve}
  \left\{ 
    \begin{aligned}
      J\left(\mathbf{q}, \frac{1}{\lambda} \mathbf{q} \times
        \mathbf{k}\right) & = - \frac{1}{ \lambda}
      \left(\left(q^1\right)^2 + \left(q^2\right)^2 \right)
      = \frac{\lambda^4 - 1}{ \lambda}=:j(\lambda)\\
      H\left(\mathbf{q}, \frac{1}{\lambda} \mathbf{q} \times
        \mathbf{k}\right) & = \frac{1}{2
        \lambda^2}\left(\left(q^1\right)^2 + \left(q^2\right)^2
      \right) + q^3 = \frac{1- \lambda^4}{2 \lambda^2} - \lambda^2=
      \frac{1-3 \lambda^4}{2 \lambda^2}=:h(\lambda)
    \end{aligned}
  \right.
\end{equation} 
for $0<|\lambda|<1$ which is a parametric curve with two branches.  We
can give an Cartesian equation for this curve. An analysis of
$h(\lambda)$ for $0<|\lambda|\leq 1$ shows that $h(\lambda)\geq
-1$. Eliminating $\lambda $ yields the two branches
$$j(h) = \pm
\frac{2}{9} \left(3-h^2 + h\sqrt{h^2+3}\right)\sqrt{h+\sqrt{h^2+3}},
$$
for $h\geq -1$. The critical set is given in Figure
\ref{critical_set_spherical_pendulum.figure}. The graph intersects the
horizontal momentum axis at $j=\pm 2 \sqrt[4]{3}/3$ and at the lower
tip $(j,h)=(0,-1)$ the graph is not smooth.

\begin{figure} [h]
  \centering
  \includegraphics[width=4cm]{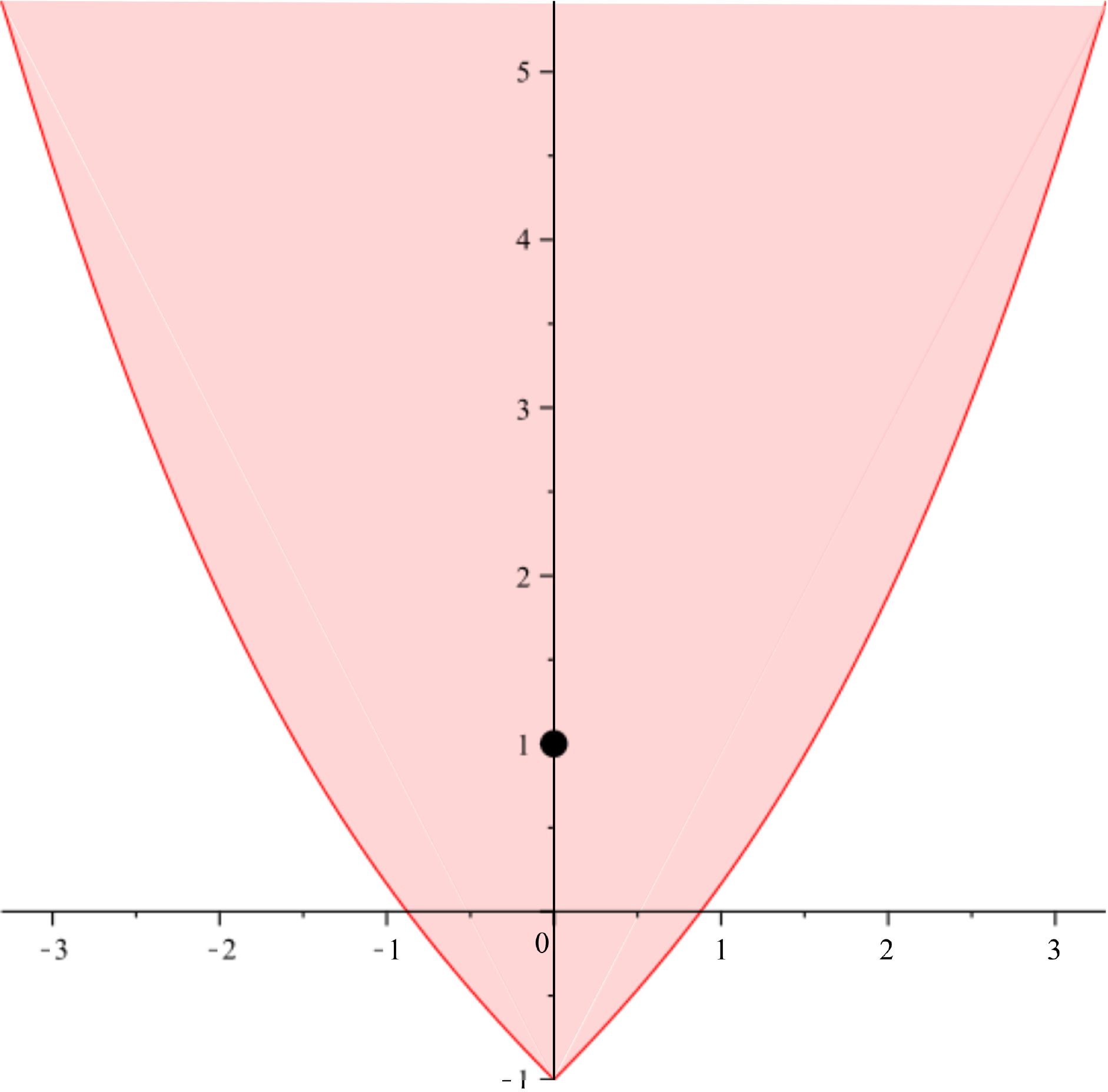}
  \caption{Image the momentum-energy map for the spherical
    pendulum. The bifurcation set $\Sigma_{F}$ is its
    boundary plus the focus-focus value $(0,1)$.}
  \label{critical_set_spherical_pendulum.figure}
\end{figure}

Since $F$ is proper, the bifurcation set equals the set of critical
values of the system, see Figure
\ref{critical_set_spherical_pendulum.figure}.

\paragraph{{\em {\rm (6)} The fibers of $F$ are connected.}}
This follows from item (5) and Theorem \ref{theo:fibers}.

Compare this with the result \cite[page 160, Table 3.2]{CuBa1997}
whose proof is quite difficult (however, it also gives the description
of the fibers of $F$).

\paragraph{{\em {\rm (7)} Range of the momentum-energy set.}}  The
range of the momentum-energy set is the epigraph of the critical set
shown in Figure \ref{critical_set_spherical_pendulum.figure}: this
follows from item~\textup{(iv)} in Theorem~\ref{theo:interior}.

\end{proof}

\section{Final remarks} \label{sec:remarks}

\subsection*{Real solutions sets}

To motivate further our connectivity results (Theorem~\ref{theo:main-connectivity0} and
Theorem~\ref{theo:main-connectivity}), consider on the real
plane with coordinates $(x,\,y)$ the following question: \emph{is the
  solution set of the polynomial equation $\,(x^2-1)\, y^2\,=\,0\,$
  connected?}  Surely the answer is yes, since the solution set
consists of two parallel vertical lines and an intersecting horizontal
line: $\{x=-1\} \cup \{x=1\} \cup \{y=0\}$.  One can modify this
equation slightly to consider the equation
$\,(x^2-1)\,(y^2+\epsilon^2)\,=\, 0\,$, $\epsilon \neq 0$.  In this
case the solution set $\{x=-1\} \cup \{x=1\}$ is disconnected, so a
small perturbation of the original equation leads to a disconnected
solution set (see Figure \ref{fig:solution}).  As is well known in
real algebraic geometry, the connectivity question is not stable under
small perturbations, so any technique to detect connectivity must be
sensitive to this issue.  

Although in all of these examples the
answers are immediate, one can easily consider equations for which
answering this connectivity question is a serious challenge.

\begin{figure}[h]
  \centering
  \includegraphics[width=0.6\textwidth]{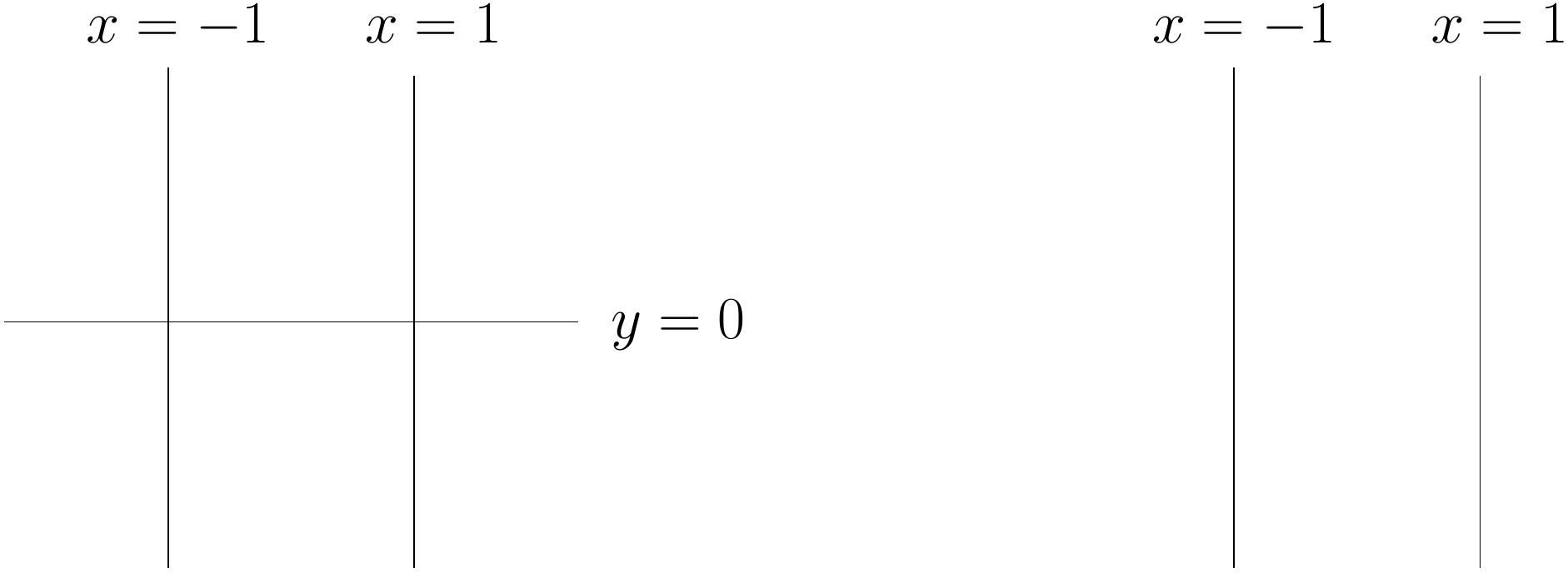}
  \caption{Solution sets of the polynomial equations $(x^2-1) y^2=0$
    and $(x^2-1)(y^2+\epsilon^2)=0$, $\epsilon\neq 0$.}
  \label{fig:solution}
\end{figure}

This question may be put in a general framework in two alternative,
equivalent, ways. Consider functions $f_1,\,\,\ldots,\,\,f_n\,\,
\colon\,\, M \subseteq \mathbb{R}^{m} \to \mathbb{R}$ and constants
$\lambda_1,\, \ldots,\, \lambda_n \in \mathbb{R}$, where $M$ is a
connected manifold. Is the {solution set} $\mathcal{S}
\subset\mathbb{R}^m$ of
\begin{eqnarray} \label{fiberequation} \left\{ \begin{array}{rl}
      f_1\,(x_1,\,\ldots,\,x_{m})&\,\,=\,\,\,\,\lambda_1 \\
      \textup{\,} \vdots \\
      f_n\,(x_1,\,\ldots, \,x_{m})&\,\,=\,\,\,\, \lambda_n
    \end{array}   
  \right.
\end{eqnarray}
a connected subset of $M$? Equivalently, are the fibers
$F^{-1}(\lambda_1,\, \ldots, \, \lambda_n)$ of the map $F\colon M
\,\subseteq\, \mathbb{R}^m \longrightarrow \mathbb{R}^n$ defined by
$$
F(x_1,\,\ldots,\,x_m)\,\,:=\,\,(f_1(x_1,\,\ldots,\,x_m),\,\,\,
\ldots\,\,\,,\,\,\,f_n(x_1,\,\ldots,\,x_m))
$$
connected? If $F$ is a scalar valued function (i.e., if there is only
one equation in the system) a well-known general method exists to
answer it when $F$ is smooth, namely, \emph{Morse-Bott theory}. 
We saw in Proposition~\ref{prop:levelset} that if $M$ is 
a connected smooth manifold and $f:M\fleche\RM$ is a
  proper Morse-Bott function whose indices and co-indices are always
  different from 1, then the level sets of $f$ are connected.

In order to motivate the idea of this result further consider the
following example: the height function defined on a $2$-sphere, and
the same height function defined on a $2$-sphere in which the North
Pole is pushed down creating two additional maximum points and a
saddle point, as in Figure~\ref{fig:morse}. When the height function
is considered on the $2$-sphere, it has connected fibers. On the other
hand, when it is considered on the right figure, many of its fibers
are disconnected. The ``essential'' difference between these two
examples is that in the second case the function has a saddle point,
while in the first case it doesn't. A saddle point has index 1.

\begin{figure}[h]
  \centering
  \includegraphics[height=5cm, width=9cm]{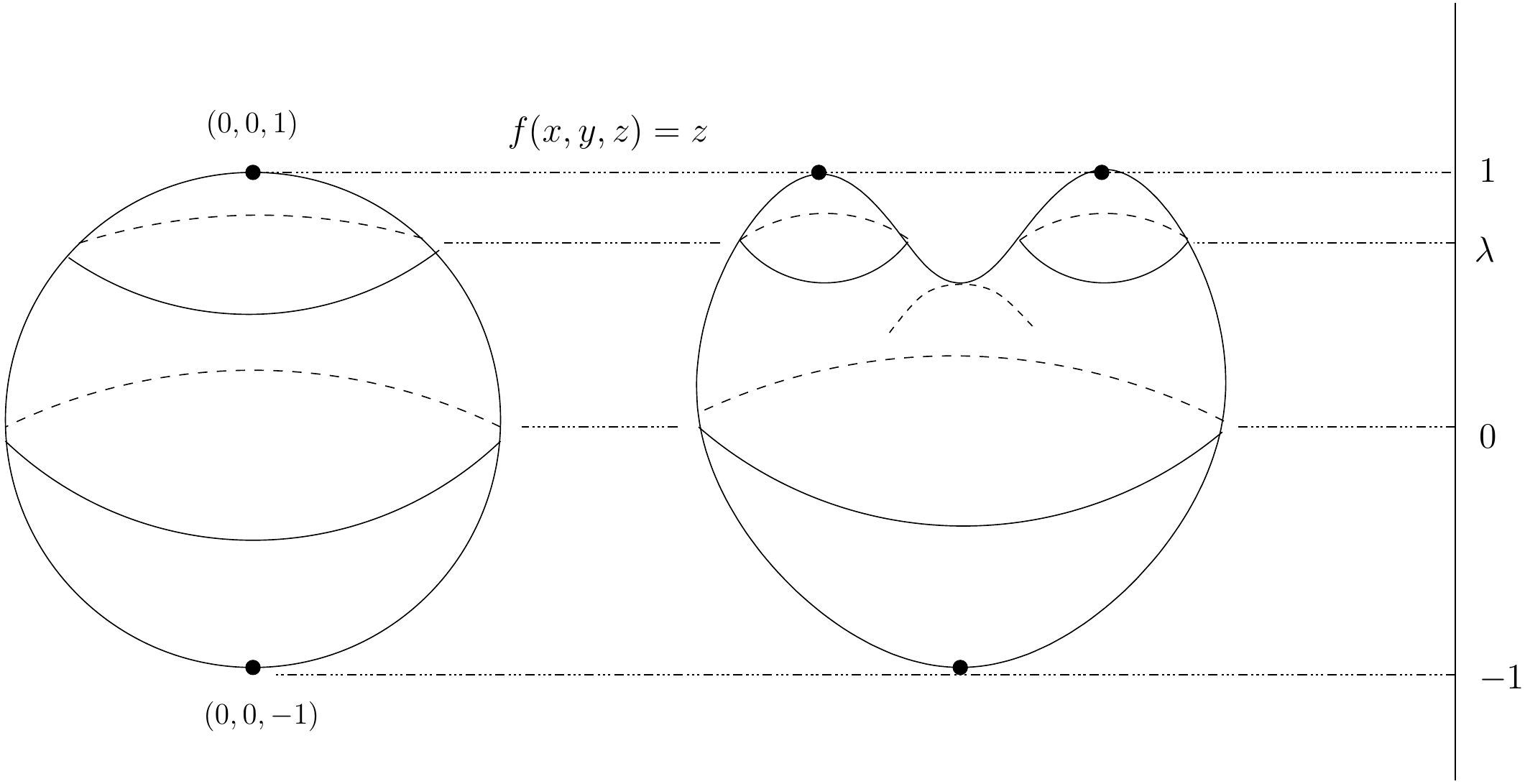}
  \caption{Height function.}
  \label{fig:morse}
\end{figure}

Next, let us look at an example of a vector valued function.  Consider
$M=S^2\times\R^2 \subset \mathbb{R}^5$ with coordinates
$(x,\,y,\,z,\,u,\, v)$.  Is the solution set $\mathcal{S}$ of
  $$
  \left\{ \begin{array}{rl}
      2u^2+2v^2 + z  &\,\,=\,\,\,\, 1 \\
      ux+vy &\,\,=\,\,\,\, 0
    \end{array}   
  \right.
$$
connected? In other words, is the set $F^{-1}(1,\,0)$ connected, where
$ F\colon S^2\times\R^2 \to \mathbb{R}^2$ defined by
$$F(x,y,z,u,v):=(2u^2+2v^2 + z,\,\, ux+vy)?$$
One is tempted to again use Morse-Bott theory to check this
connectivity, but Morse-Bott theory for which function? \emph{The first goal
of the present paper was to give a method to answer connectivity questions of
this type in the case $F$ is an integrable system.} In the paper we introduced
a method to construct Morse-Bott functions which, from the point of view of
symplectic geometry, behave well near singularities.  We saw how the behavior
of an integrable system near the singularities has a strong effect on the global properties of the 
system.

The fibers of $F\colon S^2\times\R^2 \to \mathbb{R}^2$
are, by the Theorem \ref{theo:main-connectivity}, connected, because
its image can be easily seen to be given by
Example~\ref{fig:example}. The question answered by
Theorem~\ref{theo:main-connectivity} is probably the most basic
\emph{topological} question to ask about a solution set (except for
its non-emptyness).

\subsection*{Semiclassical quantization}

Semiclassical quantization is a strong motivation for undergoing the 
systematic study of integrable systems in this paper.  Consider the quantum
version of the connectivity problem.  The functions $f_1,\dots,f_n$
are replaced by ``quantum observables'', i.e., self-adjoint operators
$\hat{f}_1,\dots,\hat{f}_n$ on a Hilbert space $\mathcal{H}$, and the
system of fiber equations \eqref{fiberequation} is replaced by the
system
\[
\left\{
  \begin{aligned}
    \hat{f}_1 \psi &= \lambda_1 \psi\\
    \vdots&\\
    \hat{f}_n \psi &= \lambda_n \psi
  \end{aligned}\right. \qquad \psi\in\mathcal{H}.
\]
So $\lambda_i$ should be an eigenvalue for $\hat{f}_i$ and the
eigenvector $\psi$ should be the same for all $i=1,\dots,n$.  There is
a chance to solve this when the operators $\hat{f}_i$ pairwise
commute. This is the quantum analogue of the Poisson commutation
property for the integrable system given by $f_1,\dots,f_n$.

In the last thirty years, semiclassical analysis has pushed this idea
quite far. It is known that to any regular Liouville torus of the
classical integrable system, one can associate a \emph{quasimode} for
the quantum system, leading to approximate
eigenvalues~\cite{CDV,lazutkin}. Thus, the study of  bifurcation
sets (described in  Theorem~\ref{theo:main-image0} and 
Theorem~\ref{theo:main-image}) is fundamental for a good understanding of the quantum
spectrum. More recently, quasimodes associated to singular fibers have
been constructed; see~\cite{san-mono} and the references therein. In some
case, it can be shown that the quantum spectrum completely determines
the classical system, see Zelditch \cite{Ze1998}.

When the system has symmetries, the quantum spectrum typically
exhibits degenerate (or almost degenerate) eigenvalues. In the setting
of integrable systems, this is revealed by the non-connectedness of
some fibers. Thus, connectivity results are of primary importance from
the quantum perspective. They predict almost degenerate (or
\emph{clusters of}) eigenvalues.  For instance, in the inverse
spectral result of~\cite{san-mono}, detecting the number of connected
components was the key point (and the most subtle) of the analysis.

\subsection*{Singular Lagrangian fibrations}

Theorem~\ref{theo:main-connectivity0} and
Theorem~\ref{theo:main-connectivity} may have applications to mirror
symmetry and symplectic topology. An integrable system without
hyperbolic singularities gives rise to a toric fibration with
singularities. The base space is endowed with a singular integral
affine structure.  Remarkably, these singular affine structures are of
key importance in various parts of symplectic topology, mirror
symmetry, and algebraic geometry; for example they play a central role
in the work of Kontsevich and Soibelman \cite{KS}.  These singular
affine structures have been studied in the context of integrable
systems (in particular by Nguy{\^e}n Ti{\^e}n Zung~\cite{zung-I}), but
also became a central concept in the works by Symington \cite{Sy2010}
and Symington-Leung \cite{ls} in the context of symplectic geometry
and topology, and by Gross-Siebert, Casta{\~n}o-Bernard,
Casta\~no-Bernard-Matessi \cite{grs1, grs2, grs3, grs4, castano-1,
  castano0, castano}, among others, in the context of mirror symmetry
and algebraic geometry. Fiber connectivity is (usually) assumed in
theorems about Lagrangian fibrations, so the theorem above gives a
method to test whether a given result in that context applies to a
Lagrangian fibration arising from an integrable system. See also
Eliashberg-Polterovich \cite[page 3]{eliashberg} for relevant examples
in the theory of symplectic quasi-states.

\subsection*{Toric and semitoric systems}

Non-degenerate integrable systems without hyperbolic singularities are
prominent in the literature, both in mathematics and in physics; see
for instance \cite{Ba2009, Sy2010,VN2007}.  They are usually called
\emph{almost-toric systems}, although the choice of name is somewhat
misleading as they may not have any periodicity from a circle or torus
action; on the other hand, these systems retain some properties of
toric systems, at least in the compact case. The most well studied
case of almost-toric system is that of toric systems coming from a
Hamiltonian torus actions.

A proof of Theorem \ref{theo:main-connectivity} for the so called
``semitoric'' systems was given by the third author in \cite{VN2007}
using the theory of Hamiltonian circle actions.  Semitoric systems are
integrable systems on four-manifolds which have no hyperbolic
singularities and for which one component of the system generates a
proper periodic flow. The proof in \cite{VN2007} uses in an essential
way the periodicity assumption and it cannot be extended, as far as we
know, to deal with the general systems treated in Theorem
\ref{theo:main-connectivity}. A classification of semitoric systems in
terms of five symplectic invariants was given by the first and third
authors in \cite{PeVN2009, PeVN2011}.

{\small \bibliographystyle{new}
  
}

\noindent
\\
\'Alvaro Pelayo \\
School of Mathematics\\
Institute for Advanced Study\\
Einstein Drive\\
Princeton, NJ 08540 USA.
\\
\\
and
\\
\\
\noindent
Washington University,  Mathematics Department \\
One Brookings Drive, Campus Box 1146\\
St Louis, MO 63130-4899, USA.\\
{\em E\--mail}: \texttt{apelayo@math.wustl.edu}

\medskip\noindent

\smallskip\noindent
Tudor S. Ratiu\\
Section de Math\'ematiques and Bernoulli Center\\
Station 8\\
Ecole Polytechnique F\'ed\'erale
de Lausanne\\
CH-1015 Lausanne, Switzerland\\
{\em E\--mail}: \texttt{tudor.ratiu@epfl.ch}

\medskip\noindent

\noindent
V\~u Ng\d oc San\\
Institut Universitaire de France\\
\\
and\\
\\
Institut de Recherches Math\'ematiques de Rennes\\
Universit\'e de Rennes 1\\
Campus de Beaulieu\\
F-35042 Rennes cedex, France\\
{\em E-mail:} \texttt{san.vu-ngoc@univ-rennes1.fr}

\end{document}